\documentclass[12pt]{amsart}
\usepackage{amsthm,amsmath,amssymb,amsopn}
\usepackage{graphicx,epsfig,subfig}
\usepackage[colorlinks,linktocpage,linkcolor=black]{hyperref}
\usepackage[table,xcdraw]{xcolor}
\usepackage{url,color}
\usepackage{booktabs}
\usepackage{colortbl,multirow}
\usepackage{appendix}
\numberwithin{equation}{section}
\numberwithin{table}{section}
\numberwithin{figure}{section}

\makeatletter
\let\c@table\c@figure
\makeatother

\newtheorem{theorem}{Theorem}[section]
\newtheorem{lemma}[theorem]{Lemma}

\newtheorem{remark}{Remark}[section]

\theoremstyle{definition}

\newcommand{\wt}[1]{\widetilde{#1}}

\newcommand{\MTh}{\mathcal{T}_h}

\newcommand{\lj}{[ \hspace{-2pt} [}
\newcommand{\rj}{] \hspace{-2pt} ]}
\def\na{\nabla}

\def\Lam{\Lambda}
\def\Om{\Omega}
\newcommand{\mb}[1]{\mathbb{#1}}
\newcommand{\mc}[1]{\mathcal{#1}}
\newcommand{\abs}[1]{\left\lvert#1\right\rvert}
\newcommand{\nm}[2]{\|\,#1\,\|_{#2}}
\newcommand{\snm}[2]{\abs{\,#1\,}_{#2}}
\newcommand{\Lr}[1]{\left(#1\right)}
\newcommand{\jump}[1]{[\![#1]\!]}
\newcommand{\aver}[1]{\left\{\,#1\,\right\}}
\newcommand{\set}[2]{\{\,#1\,\mid\,#2\}}
\newcommand{\enernm}[1]{\|\!|\,#1\,\|\!|}

\newcommand{\bm}{\boldsymbol}
\newcommand{\un}{\boldsymbol{\mathrm{n}}}

\def\dx{\mathrm{d}x}
\def\ds{\mathrm{d}s}

\definecolor{lightgray}{gray}{0.9}

\begin{document}
\title[Eigenvalue Problem]{Solving Eigenvalue Problems in a
  Discontinuous Approximation Space by Patch Reconstruction}

\author[R. Li]{Ruo Li}
\address{CAPT, LMAM and School of Mathematical Sciences, Peking
University, Beijing 100871, P. R. China} \email{rli@math.pku.edu.cn}

\author[Z.-Y. Sun]{Zhiyuan Sun}
\address{Institute of Applied Physics and Computational Mathematics,
Beijing 100094, P. R. China} \email{zysun.math@gmail.com}

\author[F.-Y.Yang]{Fanyi Yang}
\address{School of Mathematical Sciences, Peking University, Beijing
  100871, P. R. China} \email{yangfanyi@pku.edu.cn}

\begin{abstract}
   We  adapt a symmetric interior penalty discontinuous
   Galerkin method using a patch reconstructed approximation space to
   solve elliptic eigenvalue problems, including both second and
   fourth order problems in 2D and 3D. It is a direct extension of the
   method recently proposed to solve corresponding boundary value
   problems, and the optimal error estimates of the approximation to
   eigenfunctions and eigenvalues are instant consequences from
   existing results. The method enjoys the advantage that it uses only
   one degree of freedom on each element to achieve very high order
   accuracy, which is highly preferred for eigenvalue problems as
   implied by Zhang's recent study [J. Sci. Comput. 65(2), 2015]. By
   numerical results, we illustrate that higher order methods can
   provide much more reliable eigenvalues. To justify that our method
   is the right one for eigenvalue problems, we show that the patch
   reconstructed approximation space attains the same accuracy with
   fewer degrees of freedom than classical discontinuous Galerkin
   methods.  With the increasing of the polynomial order, our method
   can even achieve a better performance than conforming finite
   element methods, such methods are traditionally the methods of
   choice to solve problems with high regularities.

  \noindent\textbf{keyword:} elliptic eigenvalue problem,
  discontinuous Galerkin method, patch reconstruction
  
  
\end{abstract}

\maketitle
\section{Introduction}\label{sec:intro}
In this paper, we consider the numerical method for
solving eigenvalue problems of $2p$-th order elliptic operator for
$p=1$ and $2$. Those problems arise in many important applications.
The Laplace eigenvalue problem occurs naturally in vibrating elastic
membranes, electromagnetic waveguides and acoustic theory, and the
biharmonic eigenvalue problem appears in mechanics and inverse
scatting theory.

The conforming finite element method (FEM) for eigenvalue problems has
been well investigated. We refer to the review papers of Kuttler and
Sigillito \cite{kuttler1984eigenvalues} and Boffi
\cite{boffi2010finite} for the details. For the biharmonic operator,
we have the commonly used $C^1$ Argyris element \cite{argyris1968tuba}
and the $C^0$ interior penalty Galerkin method ($C^0$ IPG)
\cite{engel2002continuous, brenner2011c, brenner2015C0}. An old but
hot topic for eigenvalue problems is the upper and lower bounds since
\cite{forsythe1954asymptotic}. It is well known that the conforming
FEM can easily achieve the upper bound of the eigenvalues. In
\cite{armentano2003mass} and \cite{hu2004lower}, the lower bound was
achieved by mass lumping, see also other methods in
\cite{liu2013verified, boffi2010finite, armentano2004asymptotic}. Hu
et al. \cite{hu2014lower1, hu2014lower2, hu2016guaranteed} proposed a
systematic method to produce lower bounds by nonconforming
approximation spaces. The discontinuous Galerkin (DG) method, see for
example \cite{cockburn2000development, arnold2002unified,
brenner2008locally}, has been applied to the Laplace eigenvalue
problem \cite{antonietti2006discontinuous} and the Maxwell eigenvalue
problem~\cite{hesthaven2003high, warburton2006role}. As a
nonconforming approximation, the DG method admits the totally
discontinuous polynomial space which leads to a great flexibility
though it is challenged \cite{hughes2000comparison} on its efficiency
in number of degrees of freedom (DOF).

In a recent work \cite{zhang2015how}, Zhang studied an interesting
issue on the number of "trusted" eigenvalues by finite element
approximation for the elliptic eigenvalue problems. It was pointed out
therein that only eigenvalues lower in the spectrum can
achieve optimal convergence rate. Furthermore, the percentage of
reliable eigenvalues will decrease on a finer mesh even if we relax
the convergence rate to linear. Typically, the optimal convergence
rate of the elliptic eigenvalue problem is $h^{2(m+1-p)}$, where $m$
is the polynomial degree. It is implied that high order
methods are more likely to provide a greater number of reliable
eigenvalues, measured relatively to the DOFs used, than a lower order
method. 

Motivated by Zhang's result, in this paper we aim to
apply a symmetric interior penalty discontinuous Galerkin method to
elliptic eigenvalue problems. The method adopts a discontinuous
approximation space proposed in \cite{li2016discontinuous}, where it
was applied to solve elliptic boundary value problems. The core of the
method is to construct an approximation space by the patch
reconstruction technique in a way that one DOF is used in each
element. The reconstructed space is a piecewise polynomial space and
is discontinuous across the element face, thus it is a subspace of the
traditional DG space. The idea has been applied smoothly to the
biharmonic equation \cite{li2017discontinuous} and the Stokes equation
\cite{li2018finite, li2018discontinuous}. For elliptic eigenvalue
problems, it is a direct extension of the method for boundary value
problems. Consequently, the optimal error estimates of the
approximation to eigenfunctions and eigenvalues can be obtained
instantly from existing results for arbitrary order accuracy.

We present all details on the numerical results to verify that higher
order methods can provide much more reliable eigenvalues, which
perfectly agrees with the theoretical prediction in
\cite{zhang2015how}. In comparison to the classical DG method, one may
see that the patch reconstructed approximation space attains the same
accuracy with much less degrees of freedom. In case of using higher
order polynomials, the numerical results show that a better efficiency
in number of DOFs can be achieved by our method even than conforming
finite element methods. We note that for problems with high
regularities, the conforming finite element methods
traditionally outperform the other methods in number of DOFs. The new
observation here in efficiency gives us an enthusiastic encouragement
to apply our method with high order polynomials to elliptic eigenvalue
problems.

The rest of this paper is organized as follows. To be self-contained,
we describe in section \ref{sec:basis} the detailed process to
construct the approximation space and the approximation properties of
the corresponding space. The symmetric interior penalty method for
elliptic operators is presented in section \ref{sec:weakform}, and
the optimal error estimates are then given for the eigenvalues and
eigenfunctions. In section \ref{sec:examples}, we present the
numerical results to illustrate that the proposed method is efficient
for elliptic eigenvalue problems.


\section{Approximation Space}\label{sec:basis}
Let us consider a convex polygonal domain $\Om$ in $\mb{R}^D$, $D=2,3$.
$\MTh$ is a polygonal partition of the domain $\Om$. For each polygon
$K$, $h_K$ and $\abs{K}$ denote its diameter and area, respectively.
Besides, let $h{:}=\max_{K\in\MTh}h_K$. For the optimal convergence
analysis, the partition $\MTh$ is assumed to satisfy some shape
regularity conditions. Those regularity conditions are commonly used
in mimetic finite difference
schemes~\cite{Brezzi:2009,DaVeiga2014,Cangiani2011Convergence} and
discontinuous Galerkin method~\cite{Mu:2014}, which are stated as
follows:
\begin{enumerate}
\item[{\bf A1}\;]Any element $K \in \MTh$ admits a sub-decomposition
  $\wt{\MTh}|_K$ that consists of at most $N_s$ triangles, where $N_s$
  is an integer independent of $h$;
\item[{\bf A2}\;]If all the triangles $T\in\wt{\MTh}$ are
  shape-regular in the sense of
  Ciarlet-Raviart~\cite{ciarlet2002finite}: there exists a real
  positive number $\sigma$ independent of $h$ such that
  $h_T/\rho_T\le\sigma$, where $\rho_T$ is the radius of the largest
  ball inscribed in $T$. Then the $\wt{\MTh}$ is a compatible
  sub-decomposition.
\end{enumerate}
The above regularity assumptions lead to some useful estimates, such
as Agmon inequality, approximation property and inverse inequality.
Those inequalities are the foundations to derive the approximation
error estimates for the finite element method. We refer
to~\cite{li2016discontinuous} for the detailed discussion.

The reconstruction operator $\mc{R}$ can be constructed with the given
partition $\MTh$. The degrees of freedom of $\mc{R}$ are located at
one point $x_K\in K$ on each element which are called the sampling
nodes or collocation points. We usually assign the barycenter of $K$
as the sampling node $x_K$. Furthermore, the reconstruction operator
$\mc{R}$ is defined element-wise. An element patch denoted as $S(K)$
is constructed for each element $K$. $S(K)$ is an agglomeration of
elements including $K$ itself and other elements nearby $K$. Let
$\mc{I}_K$ denote the set of sampling nodes belonging to $S(K)$, $\#
S(K)$ and $\# \mc{I}_K$ denote the number of elements belonging to
$S(K)$ and the number of sampling nodes belonging to $\mc{I}_K$,
respectively. Obviously, these two numbers are equal to each other. We
define $d_K{:}=\text{diam}\;S(K)$ and $d{:}=\max_{K\in\MTh}d_K$.

Here we specify the way to construct the element patch while it can be
quite flexible, see~\cite{li2012efficient,li2016discontinuous} for the
alternative approaches. First, a constant number $t$ is assigned to
$\# S(K)$ which is determined by the degree of polynomials. Then we
initialize $S(K)$ as $\{ K \}$, and fill $S(K)$ by adding the nearest
Von Neumann neighbor (adjacent edge-neighboring elements) of the
current geometry $S(K)$. We terminate the recursive process until the
number $\# S(K)$ reaches the number $t$. With such an approach, the
element patches are obtained with a constant number, which is
convenient for the implementation. Meanwhile, the shape regularity of
the geometry of $S(K)$ preserves. All the sampling nodes $x_K$ are
located in element $K$ and all element patches are connected set, that
the stability of reconstruction is fair promising. The reconstruction
process can be conducted element-wise after the sampling nodes
$\mc{I}_K$ and element patch $S(K)$ are specified.

Let $U_h$ be the piecewise constant space associated with $\MTh$,
i.e.,
\[
U_h{:}=\set{v\in L^2(\Om)}{v|_K\in\mb{P}^0(K), \ \forall K \in \MTh}.
\] 
For a piecewise constant function $v\in U_h$ and an element $K$, a
high-order approximation polynomial $\mc{R}_{K} v$ of degree $m$ can
be obtained by solving the following discrete local least-squares:
\begin{equation}\label{eq:leastsquares}
\mc{R}_{K} v=\arg
\min_{p\in\mb{P}^m(S(K))}\sum_{x\in\mc{I}_K}\abs{v(x)-p(x)}^2.
\end{equation}
We assume the problem \eqref{eq:leastsquares} has a unique solution
\cite{li2016discontinuous}. Now, we concentrate on the reconstruction
operator and the corresponding finite element space. Although
$\mc{R}_K v$ gives an approximation polynomial on element patch
$S(K)$, we only use it on element $K$. The global reconstruction
operator $\mc{R}$ is defined as:
\[
(\mc{R} v)|_K:= (\mc{R}_{K} v)|_K, \quad \forall K \in \MTh.
\]
The reconstruction operator $\mc{R}$ actually defines a linear
operator which maps $U_h$ into a piecewise polynomial space,
denoted by
\[
V_h:=\mc{R}(U_h).
\] 
Here, $V_h$ is the reconstructed finite element space which is
spanned by the basis functions $\{\psi_K\}$. Here the basis functions
are defined by the reconstruction operator,
\[
\psi_K:= \mc{R} e_K,
\]
where $e_K\in U_h$ is the characteristic function corresponding to
$K$,
\begin{displaymath}
  e_{K}(x)=\begin{cases} 1,\ x \in K,\\ 0,\ x \notin K.\\
  \end{cases}
\end{displaymath}
Thereafter the reconstruction operator can be explicitly expressed
\[
\mc{R} g =\sum_{K \in \MTh} g(x_K) \psi_K(x) , \quad \forall g \in
U_h.
\]

We present a 3D example below to illustrate the implementation of
reconstruction process, while the details for 1D implementation and 2D
implementation can be found in \cite{li2017discontinuous} and
\cite{li2018finite}, respectively. We consider a linear reconstruction
on a cubic domain $[0,1]^3$. The domain is partitioned into
quasi-uniform tetrahedron elements using {\it
  Gmsh}~\cite{geuzaine2009gmsh}, which is shown in Figure
\ref{tetra_mesh}. We take element $K_0$ as an instance (see Figure
\ref{tetra_mesh}). The number of degrees of freedom demanded by linear
reconstruction is $4$. Therefore, the $\# S(K_0)$ could be taken as
$5$. In this case, the element patch is containing the element itself
and 4 Von Neumann neighbors coincidentally. Figure \ref{tetra_patch}
shows the geometry of the element patch and the corresponding sampling
nodes. The element patch $S(K_0)$ is chosen as
\[
S(K_0)=\left\{K_0,K_1,K_2,K_3,K_4\right\},
\]
and the sampling nodes are as follows,
\[
\mc I_{K_0} = \left\{ (x_{K_{i}}, y_{K_{i}}, z_{K_{i}}),\quad
i=0,1,2,3,4  \right\}.
\]

\begin{figure}
  \begin{center}
    \includegraphics[width=0.48\textwidth]{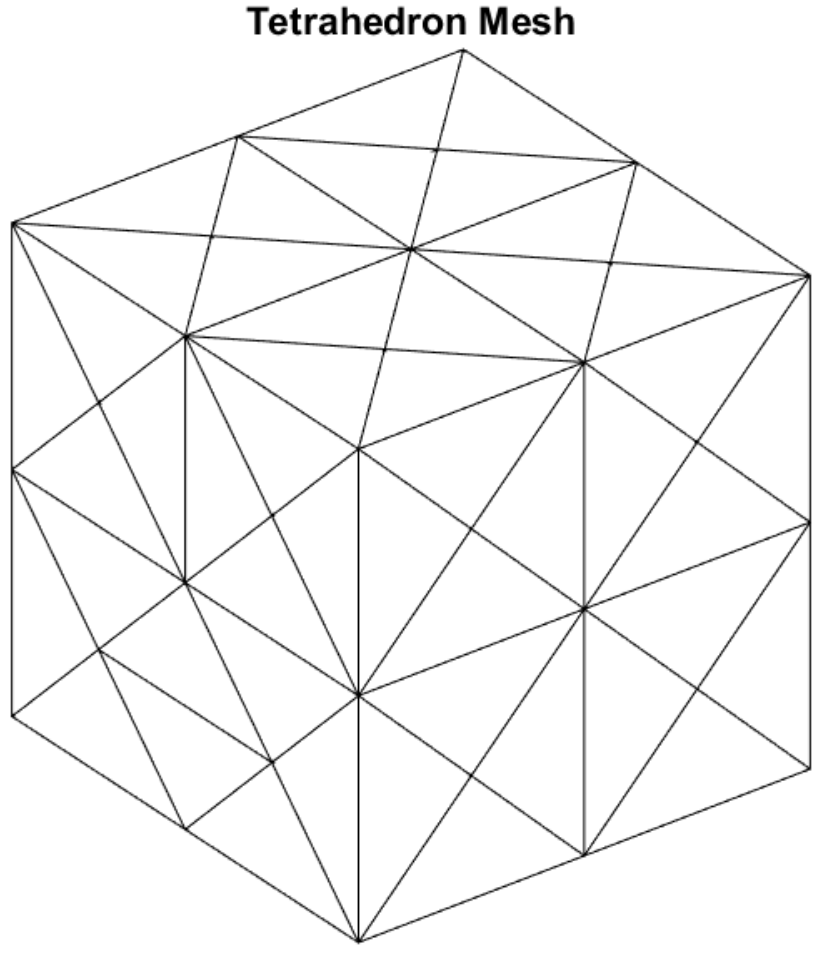}
    \includegraphics[width=0.48\textwidth]{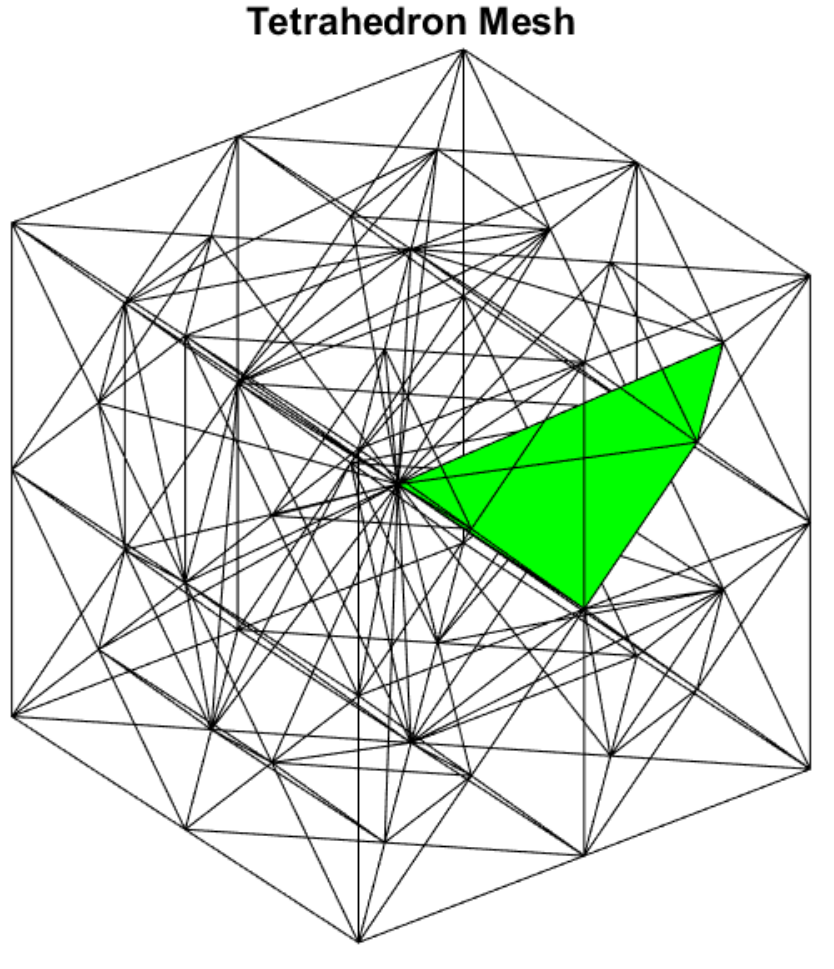}
    \caption{The tetrahedron mesh (left) and the element
    $K_0$(right).}
    \label{tetra_mesh}
  \end{center}
\end{figure}

\begin{figure}
  \begin{center}
    \includegraphics[width=0.48\textwidth]{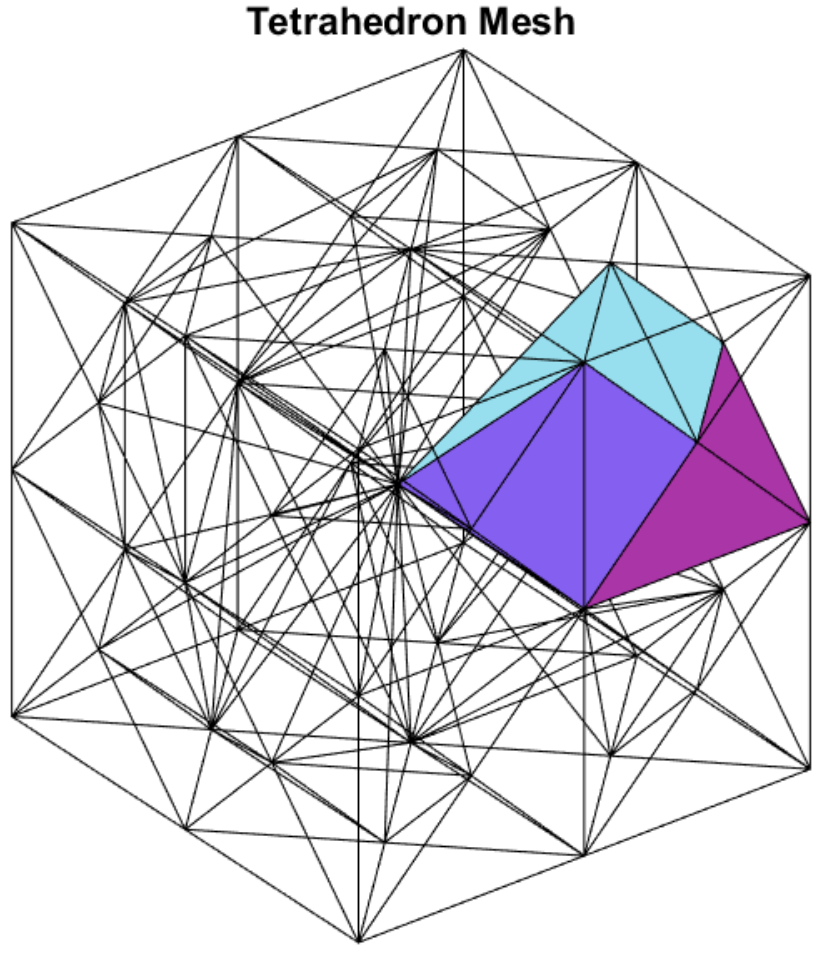}
    \includegraphics[width=0.48\textwidth]{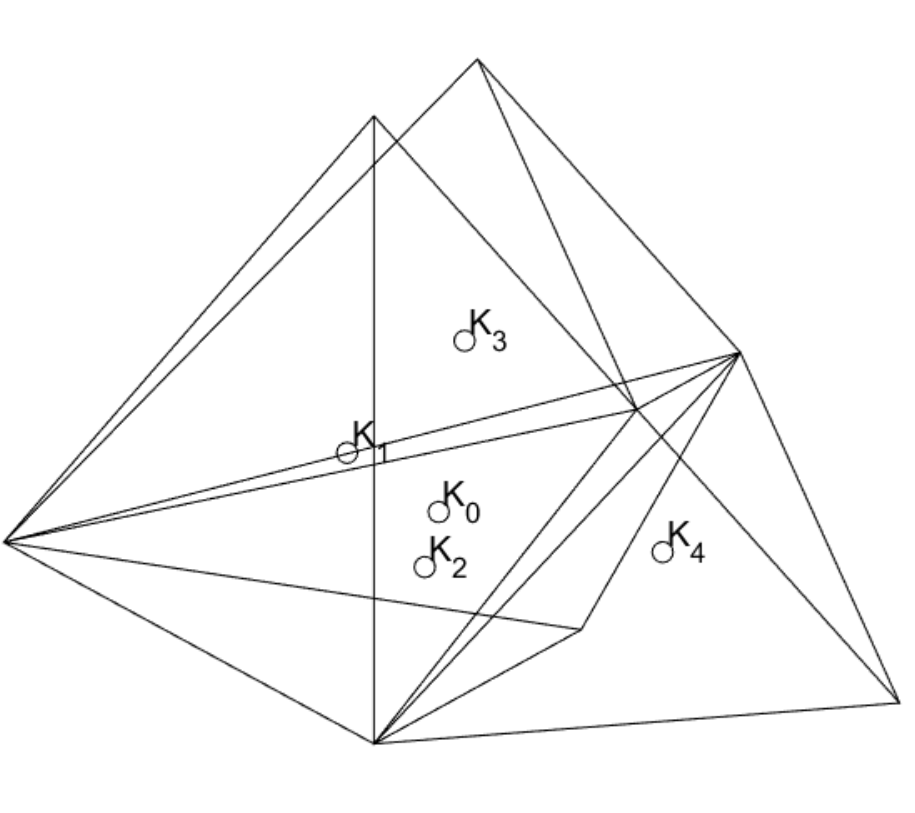}
    \caption{The shape of element patch (left) and the perspective
      view of element patch and sampling nodes (right).}
    \label{tetra_patch}
  \end{center}
\end{figure}

For any continuous function $g$,  we consider the linear
approximation for an illustration. For the polynomial degree $m =1$, 
the least
squares problem \eqref{eq:leastsquares} is specified as
\begin{displaymath}
  \mc R_{K_0} g = \mathop{\arg \min}_{ (a, b, c, d) \in \mathbb R}
  \sum_{i=0}^{4} |g(x_{K_{i}},y_{K_{i}},z_{K_{i}}) - (a + bx_{K_{i}} +
  cy_{K_{i}}+ d z_{K_{i}})|^2.
\end{displaymath}
The solution of the problem is given by the generalized inverse of
matrix,
\begin{displaymath}
[a,b,c,d]^{T}=(A^TA)^{-1}A^T q,
\end{displaymath}
where $A$ and $q$ are
\begin{displaymath}
  A = \begin{bmatrix} 1 & x_{K_{0}} & y_{K_{0}} & z_{K_{0}} \\ 1 &
    x_{K_{1}} & y_{K_{1}} & z_{K_{1}} \\ 1 & x_{K_{2}} & y_{K_{2}} &
    z_{K_{2}} \\ 1 & x_{K_{3}} & y_{K_{3}} & z_{K_{3}} \\ 1 &
    x_{K_{4}} & y_{K_{4}} & z_{K_{4}}
  \end{bmatrix}, \quad 
  q = \begin{bmatrix} g(x_{K_{0}},y_{K_{0}},z_{K_{0}})
    \\ g(x_{K_{1}},y_{K_{1}},z_{K_{1}})
    \\ g(x_{K_{2}},y_{K_{2}},z_{K_{2}})
    \\ g(x_{K_{3}},y_{K_{3}},z_{K_{3}})
    \\ g(x_{K_{4}},y_{K_{4}},z_{K_{4}})
  \end{bmatrix}.
\end{displaymath}

A direct observation is that matrix $(A^T A)^{-1}A^T$ is not relevant
to the interpolation function $g$. Moreover, the matrix $(A^T
A)^{-1}A^T$ actually stores the polynomial basis function coefficients
corresponding to $\psi_{K_i}$, $i=0, \cdots, 4$. All the basis
functions and the finite element space $V_h$ are determined after the
reconstruction process on each element $\forall K\in
\MTh$. Clearly, the basis functions are discontinuous across the
interface.

 Next, for completeness, we report the results on the
properties of the reconstruction operator. Following
\cite{li2012efficient}, we first make the following assumption.

{\bf Assumption A}\; For any $K\in\MTh$ and $g\in\mb{P}^m(S(K))$,
\begin{equation}\label{assumption:uniqueness}
g|_{\mc{I}(K)}=0\quad\text{implies}\quad g|_{S(K)}\equiv 0.
\end{equation}

This assumption implies the uniqueness for least squares problem
\eqref{eq:leastsquares}. A necessary condition for {\bf Assumption A}
is that the number $\# \mc{I}_K$ needs to be greater than
$\text{dim}(\mb{P}^m)$, whose quantities are $m+1$, $(m+1)(m+2)/2$ and
$(3m^2+3m+2)/2$ corresponding to 1D,2D and 3D, respectively. A
constant $\Lambda(m,\mathcal{I}_K)$ is defined as
\cite{li2012efficient}:
\begin{equation}\label{eq:cons}
\Lambda(m, \mathcal{I}_K){:}=\max_{p\in \mathbb{P}^m(S(K))}
\dfrac{\nm{p}{L^\infty(S(K))}}{\nm{p|_{\mc{I}_K}}{\ell_\infty}}.
\end{equation}
Then, the uniform upper bound can be obtained by adding some
constrains on element patches and the partition, see also \cite{
  li2016discontinuous} for the details. We have the following
properties of the reconstruction operator $\mc{R}_K$.
\begin{lemma}\label{theorem:localapp}
\cite[Theorem 3.3]{li2012efficient} If {\em Assumption A} holds, then
there exists a unique solution to~\eqref{eq:leastsquares}. Moreover
$\mc{R}_{K}$ satisfies
  \begin{equation}\label{eq:invariance}
    \mc{R}_Kg=g\quad\text{for all\quad}g\in\mb{P}^m(S(K)).
  \end{equation}
  The stability property holds true for any $K\in\MTh$ and $g\in
  C^0(S(K))$ as
  \begin{equation}\label{eq:continuous}
    \nm{\mc{R}_K g}{L^{\infty}(K)}\le\Lambda(m , \mc{I}_K) \sqrt{\#
      \mc{I}_K}\nm{g|_{\mc{I}(K)}}{\ell_\infty},
  \end{equation}
  and the quasi-optimal approximation property is valid in the sense
  \begin{equation}\label{eq:approximation}
    \nm{g -\mc{R}_K g}{L^{\infty}(K)}\le\Lambda_m
    \inf_{p\in\mb{P}^m(S(K))} \nm{g - p}{L^{\infty}(S(K))}, \quad
    \forall K \in \MTh,
  \end{equation}
  where $\Lambda_m{:}=\max_{K\in \MTh}
  \{1+\Lambda(m,\mathcal{I}_K)\sqrt{\# \mathcal{I}_K}\}$.
\end{lemma}
With Lemma \ref{theorem:localapp} and the interpolation result in
\cite{DupontScott:1980}, the local estimates on element $K$ can be
obtained.

\begin{lemma} \label{theorem:normapp}
  \cite[Lemma 2.4]{li2016discontinuous} Let $u \in C^0
  \left(\Omega \right) \cap H^{m+1}(\Omega)$, then there exists a
  constant $C$ that depends on $N_s$ and $\sigma$, but independent of
  $h$, such that
  \begin{equation}\label{eq:l2app}
    \nm{g-\mc{R} g}{L^{2}(K)}\le C\Lam_m h_Kd_K^m\snm{g}{H^{m+1}(K)},
  \end{equation}
  and
  \begin{equation}\label{eq:h1app}
    \nm{\na(g-\mathcal{R} g)}{L^2(K)} \le
    C\Lr{h_K^m+\Lam_{m}d_K^m}\snm{g}{H^{m+1}(K)}.
  \end{equation}
\end{lemma}

%

\section{Elliptic  Eigenvalue Problems}\label{sec:weakform}
Let us consider the $2p$-th $(p=1,2)$ order elliptic eigenvalue
problems, for $p=1$, the second order elliptic eigenvalue problem
reads:
\begin{equation}\label{eq:elliptic}
\left\{
\begin{array}{ll}
  - \Delta u = \lambda u, & \text{in } \Omega,\\ [1.5ex] u =0, &
  \text{on } \partial \Omega,
\end{array}\right.
\end{equation}
and the corresponding weak form is: find $\lambda \in \mathbb{R}$ and
$ u\in V=H^{1}_{0}(\Omega)$, with $ u \neq 0$, such that
\begin{equation*}
  a(u,v)=\lambda(u,v),\quad \forall v\in V,
\end{equation*}
where $a(u,v)=\int_{\Omega} \na u \cdot \na v\dx$ and $(u, v) :=
\int_\Omega u v \dx$.

For $p=2$, the biharmonic eigenvalue problem reads:
\begin{equation}\label{eq:harmonic}
\left\{
\begin{array}{ll}
\Delta^2 u = \lambda u, & \text{in } \ \Omega,\\[1.5ex] u
=\dfrac{\partial u}{\partial \un}=0, & \text{on } \partial \Omega,
\end{array}\right.
\end{equation}
and the corresponding weak form is: find $\lambda \in \mathbb{R}$ and
$ u\in V=H^{2}_{0}(\Omega)$, with $ u \neq 0$, such that
\begin{equation*}
  a(u,v)=\lambda(u,v),\quad \forall v\in V,
\end{equation*}
where $a(u,v)=\int_{\Omega} \Delta u \Delta v\dx$.

The discretized variational problem for equations \eqref{eq:elliptic}
and \eqref{eq:harmonic} reads: find $\lambda_h \in \mathbb{R}$ and $
u_h\in U_h$, with $ u_h \neq 0$, such that
\begin{equation}\label{eq:weak_form}
a_{h}(\mc{R} u_h,\mc{R}v_h)=\lambda_h(\mc{R} u_h, \mc{R} v_h),\quad
\forall v_h \in U_h.
\end{equation}
Here we use the notations $a$, $a_h$ for unification. In the rest of
the paper, we will specify the sense of the notation when a particular
equation is considered.

The symmetric interior penalty method is employed to discretize the
elliptic operators. For the second order elliptic operator,
$a_h(\cdot, \cdot)$ is
\begin{equation}\label{elliptic_operator}
  \begin{aligned}
    a_h(v,w){:}=&\sum_{K\in \MTh}\int_{K}\na v\cdot\na
    w\dx\\ &-\sum_{e\in \mc{E}_h}\int_{e}\Lr{\jump{\na
        v}\aver{w}+\jump{\na w}\aver{v}}\ds\\ &+\sum_{e\in
      \mc{E}_h}\int_{e}\eta_eh_e^{-1}\jump{v}\cdot\jump{w}\ds,
  \end{aligned}
\end{equation}
and for the biharmonic operator, $a_h(\cdot, \cdot)$ is
\begin{equation}\label{biharmonic_operator}
    \begin{aligned}
    a_h(v, w){:}=&\sum_{K\in \MTh}\int_{K}\Delta v\Delta
    w\dx\\ &+\sum_{e\in \mc{E}_h}\int_{e}\Lr{ \jump{v}
      \aver{\nabla\Delta w} +\jump{w} \aver{\nabla\Delta v} } \ds
    \\ &-\sum_{e\in \mc{E}_h}\int_{e}\Lr{ \aver{\Delta w}\jump{\nabla
        v}+\aver{\Delta v} \jump{\nabla w} } \ds\\ &+\sum_{e\in
      \mc{E}_h}\int_{e}\Lr{ \alpha_e h_e^{-3 }\jump{v} \cdot \jump{w}
      +\beta_e h_e^{-1} \jump{\nabla u} \jump{\nabla v} } \ds,
  \end{aligned}
\end{equation}
where $\eta_e,\alpha_e,\beta_e$ are positive constants. Here we let
$\mc{E}_h$ denote the collection of all the faces of $\MTh$,
$\mc{E}_h^i$ denote the collection of the interior faces. The set of
boundary faces is denoted as $\mc{E}_h^b$, and then
$\mc{E}_h=\mc{E}_h^i\cup \mc{E}_h^b$.  Let $e$ be an interior face
shared by two neighbouring elements $K^+, K^-$, and $\bm{\mathrm n}^+$
and $\bm{\mathrm n}^-$ denote the corresponding outward unit normal.
For the scalar-valued function $q$ and the vector-valued function
$\bm{v}$, the \emph{average} operator $\aver{\cdot}$ and the
\emph{jump} operator $\jump{\cdot}$ are defined as
\begin{displaymath}
  \{q\}=\frac12(q^++q^-), \quad\{\bm v\}=\frac12(\bm v^++\bm v^-),
\end{displaymath}
and
\begin{displaymath}
  \lj q\rj=\bm{\mathrm n^+}q^++\bm{\mathrm n^-}q^-, \quad\lj\bm v\rj=
  \bm{\mathrm n^-}\cdot\bm v^++\bm{\mathrm n^-}\cdot\bm v^-.
\end{displaymath}
Here $q^+=q|_{K^+}$, $\bm v^+=\bm v|_{K^+}$ and $q^-=q|_{K^-}$, $\bm
v^-=\bm v|_{K^-}$. For $e\in\mc{E}_h^b$, we set
\begin{displaymath}
  \{q\}=q|_{K},\quad \lj q\rj=\bm{\mathrm n}q|_{K},
\end{displaymath}
and
\begin{displaymath}
  \{\bm v\}=\bm v|_{K},\quad \lj\bm v\rj=\bm{\mathrm n}\cdot\bm
  v|_{K}.
\end{displaymath}

We note that the problem \eqref{eq:weak_form} is equivalent to the
following problem: find $\lambda_h \in \mathbb{R}$ and $\varphi_h\in
V_h$, with $ \varphi_h \neq 0$, such that
\[
a_h(\varphi_h,\psi_h)=\lambda_h(\varphi_h,\psi_h),\quad \forall
\psi_h\in V_h.
\]
This is a more standard formulation for finite element methods. By the
formulation \eqref{eq:weak_form}, it is emphasized that the number of
DOFs of the approximation space is always $\text{dim}(U_h)$.

We define the energy norms $\|\cdot\|_{h}$ and $\enernm{\cdot}_{h}$
for any $v\in V_h=\mc{R}(U_h)$ as:
\begin{equation}\label{eq:energy_norms}
\begin{aligned}
\|v\|_{h}^2&=\sum_{K\in \MTh}\nm{\na v}{L^2(K)}^{2} + \sum_{e\in
  \mc{E}_h}h_e^{-1}\nm{\jump{v}}{L^2(e)}^2,\\ 
  \enernm{v}_{h}^2&=\sum_{K\in \MTh}\nm{\Delta v}{L^2(K)}^{2} +
  \sum_{e\in \mc{E}_h}h_e^{-3}\nm{\jump{v}}{L^2(e)}^2 + \sum_{e\in
  \mc{E}_h}h_e^{-1}\nm{\jump{\na v}}{L^2(e)}^2.
\end{aligned}
\end{equation}

From the Lemma \ref{theorem:normapp} and Agmon inequality, the
following interpolation estimates are straightforward results for
the reconstruction operator in the energy norm.

\begin{lemma}\label{lemma:approximate_energy_error}
  \cite[Equation 3.4]{li2016discontinuous} \cite[Theorem
    2.1]{li2017discontinuous} Let $u\in H^{m+1}(\Omega)$, and
      $\mc{R}u\in V_h$ be the interpolation polynomial of $u$, there
      exists a constant $C$ that depends on $N_s$, $\sigma$ and $m$,
      but independent of $h$, such that
  \begin{equation}\label{eq:approx_energy_error}
    \begin{split}
      \|u-\mc{R}u \|_h \leq & C (h^{m}+\Lambda_{m}d^{m}) | u
      |_{H^{m+1}(\Omega)} ,\\ \enernm{u-\mc{R}u}_{h}\leq & C
      (h^{m-1}+\Lambda_{m}d^{m-1}) | u |_{H^{m+1}(\Omega)}.
    \end{split}
  \end{equation}
\end{lemma}

Next, the boundedness and coercivity of the bilinear operator
$a_h(\cdot,\cdot)$ in \eqref{elliptic_operator} and
\eqref{biharmonic_operator} are as below.

\begin{lemma}\label{lemma:bilinear_operator_property}
\cite[Equations 4.4,4.10]{arnold2002unified} If the penalty
constant $\eta_e$ is sufficiently large, then the bilinear operator
\eqref{elliptic_operator} is bounded and coercive, indeed there exist constants $C_b$
and $C_s$, such that
\begin{equation}\label{eq:elliptic_property}
\begin{split}
 a_h(\mc{R}v_h,\mc{R}v_h)&\geq C_b \|\mc{R}v_h\|_{h}^2, \quad \forall
 v_h \in U_h,\\ a_h(\mc{R}u_h,\mc{R}v_h)&\leq C_s \|\mc{R}u_h\|_{h}
 \|\mc{R}v_h\|_{h}, \quad \forall u_h, v_h \in U_h.
\end{split}
\end{equation}
\cite[Lemmata 3.1, 3.2]{li2017discontinuous} If the penalty constants
$\alpha_e,\beta_e$ are sufficiently large , then there exist constants
$C_b$ and $C_s$, such that the bilinear operator
\eqref{biharmonic_operator} satisfies
\begin{equation}\label{eq:biharmonic_property}
\begin{split}
 a_h(\mc{R}v_h,\mc{R}v_h)&\geq C_b \enernm{\mc{R}v_h}_{h}^2, \quad
 \forall v_h \in U_h,\\ a_h(\mc{R}u_h,\mc{R}v_h)&\leq C_s
 \enernm{\mc{R}u_h}_{h} \enernm{\mc{R}v_h}_{h}, \quad \forall u_h, v_h
 \in U_h.
\end{split}
\end{equation}
\end{lemma}

We refer to \cite{arnold2002unified,li2017discontinuous} for the
proof.

To derive the error estimates, we introduce the sum space
$V(h)=V+\mc{R}(U_h)$, and endows it with the energy norm
\eqref{eq:energy_norms}, denoted as $\|\cdot\|_{V(h)}$ for
unification,
\begin{displaymath}
  \|\cdot\|_{V(h)}=
  \begin{cases}
   \|\cdot\|_{h},\ p=1,\\ \enernm{\cdot}_{h},\ p=2.\\
  \end{cases}
\end{displaymath}
Let $\lambda^{(i)}$, $i\in \mathbb{N}$, denote the sequence of
eigenvalues of \eqref{eq:elliptic} and \eqref{eq:harmonic} with the
natural numbering
\begin{displaymath}
\lambda^{(1)}\leq \lambda^{(2)} \leq \cdots \leq \lambda^{(i)}\leq
\cdots,
\end{displaymath}
and the corresponding eigenfunctions with the standard normalization
$\|u^{(i)}\|=1$
\begin{displaymath}
u^{(1)},u^{(2)},\cdots,u^{(i)},\cdots, 
\end{displaymath}
which are orthogonal to each other
\begin{displaymath}
(u^{(i)},u^{(j)})=0, \quad \text{if}\ i\neq j.
\end{displaymath}

Let $N=\text{dim}(V_h)$, thus the discrete eigenvalues of
\eqref{eq:weak_form} can be ordered as follows:
\begin{displaymath}
\lambda^{(1)}_{h}\leq \lambda^{(2)}_{h} \leq \cdots \leq
\lambda^{(N)}_{h},
\end{displaymath}
and the discrete eigenfunctions with the normalization
$\|\mc{R}u^{(i)}_{h}\|=1$,
\begin{displaymath}
\mc{R}u^{(1)}_{h},\mc{R}u^{(2)}_h,\cdots,\mc{R}u^{(N)}_{h},
\end{displaymath}
which satisfy the same orthogonalities
\begin{displaymath}
(\mc{R}u^{(i)}_{h},\mc{R}u^{(j)}_{h})=0, \quad \text{if}\ i\neq j.
\end{displaymath}

The convergence analysis for the eigenvalue problem
\eqref{eq:weak_form} can be obtained by the Babu\v{s}ka-Osborn
theory~\cite{Osborn1991Eigenvalue}. We define the following continuous
and discrete solution operators:
\begin{equation}\label{eq:solution_operator}
\begin{split}
T&:L^{2}(\Omega) \rightarrow V \quad a(Tf,v)=(f,v), \ \forall v\in
V,\\ T_h&:L^{2}(\Omega) \rightarrow \mc{R}(U_h) \quad a_{h}(T_h
f,\mc{R}v)=(f,\mc{R}v), \ \forall v\in U_h.
\end{split}
\end{equation}
Obviously the operator $T$ and $T_h$ are self-adjoint and from the
elliptic regularity, there exists $\epsilon>0 $ such that
\[
\|Tf - T_hf\|_{V(h)}\leq C h^{\epsilon} \|f\|_{L^{2}(\Omega)}.
\]
And the operators have the gradual approximation property,
\begin{equation}\label{solution_operator_approx}
\lim_{h\rightarrow 0} \|T - T_h\|_{\mc{L}(V(h))}=0.
\end{equation}

Let $\sigma(T),\sigma(T_h)$ and $\rho(T),\rho(T_h)$ denote the
spectrum and the resolvent set of the solution operator $T$ and $T_h$,
respectively. Define the resolvent operators as follows
\begin{equation*}
\begin{split}
  R_z(T):=&(z-T)^{-1},\ \forall z\in \rho(T),\quad V\rightarrow
  V,\\ R_z(T_h):=& (z-T_h)^{-1},\ \forall z\in \rho(T),\quad
  \mc{R}(U_h)\rightarrow \mc{R}(U_h).
\end{split}
\end{equation*}
Then the first result of convergence is that there is no pollution of
the spectrum.
\begin{theorem}\label{Non-pollution of the spectrum}
  \cite[Theorem 9.1]{boffi2010finite} Assume the convergence in norm
  \eqref{solution_operator_approx} is satisfied, for any compact set
  $K \subset \rho(T)$, there exists $h_0>0$, such that, for all
  $h<h_0$, we have
  \[
    K \subset \rho(T_h).
  \]
If $\mu \in\sigma(T)$ is a non-zero eigenvalue with algebraic
multiplicity $k$, then exactly $k$ discrete eigenvalues of $T_h$,
convergence to $\mu$ as $h$ tend to zero.
\end{theorem}

Let $\Gamma$ be an arbitrary closed smooth curve $\Gamma \in \rho(T)$
which encloses $\mu\in\sigma(T)$, and no other elements of
$\sigma(T)$, we define the Riesz spectral projection operators $E$, $E_h$ by:
\begin{equation*}\label{eq:spectral_operator}
\begin{split}
E&:L^{2}(\Omega) \rightarrow V \quad E(\lambda)=\frac{1}{2\pi
  i}\int_{\Gamma} R_z(T)\,dz, \\ E_h&:L^{2}(\Omega) \rightarrow
\mc{R}(U_h) \quad E_{h}(\lambda)=\frac{1}{2\pi i}\int_{\Gamma}
R_z(T_h)\,dz.
\end{split}
\end{equation*}
When $h$ is sufficiently small, we have $\Gamma \in \rho(T_h)$ and
$\Gamma$ encloses exactly $k$ eigenvalues of $T_h$. More precisely,
the dimension of $E(\mu)V$ and $E_{h}(\mu)\mc{R}(U_h)$ is equal to
$k$. Further we have
\begin{equation}\label{eigenspace_approx}
\lim_{h\rightarrow 0} \|E-E_h\|_{\mc{L}(L^2(\Omega),V(h))}=0.
\end{equation}
The convergence of the generalized eigenvectors has been claimed.

The gap between the eigenspaces is defined as follows,
\begin{align*}
\delta(E,F)&=\sup_{u\in E,\|u\|=1} \inf_{v\in F} \|u-v\|
,\\ \hat{\delta}(E,F)&= \max(\delta(E,F),\delta(F,E)) .
\end{align*}

\begin{lemma}\label{non-pollution of eigenspace}
\cite[Theorem 9.3]{boffi2010finite} Let $\mu$ be a non-zero eigenvalue
of $T$, let $E =E(\mu) V$ be its generalized eigenspace, and let
$E_h=E_h(\mu) \mc{R}(U_h)$. Then
\[
\hat{\delta}(E,E_h)\leq C \|(T-T_h)_{|E}\|_{\mc{L}(V(h))}.
\]
\end{lemma}

\begin{lemma}\label{corollary_non-pollution}
\cite[Corollary 9.4]{boffi2010finite} Let $\lambda$ be a non-zero
eigenvalue of \eqref{eq:elliptic} and \eqref{eq:harmonic},
respectively, and $E =E(\lambda^{-1}) V$ be its generalized
eigenspace, and let $E_h=E_h(\lambda^{-1}) \mc{R}(U_h)$. Then
\[
\hat{\delta}(E,E_h)\leq C\sup_{u\in E,\|u\|_{V(h)}=1} \inf_{v\in U_h}
\|u-\mc{R}v\|_{V(h)}.
\]
\end{lemma}

We now claim the approximation estimate for the solution operator, and
we refer to~\cite{li2016discontinuous,li2017discontinuous} for more
details.

\begin{lemma}\label{lemma:appro_eigenvector}
Let $\lambda$ be a non-zero eigenvalue of
    \eqref{eq:elliptic} and \eqref{eq:harmonic}, respectively, let $E$
    be the eigenspace associated with $\lambda$, and its regularity
    satisfy
$E\subset H^{m+1}(\Omega)$, $m \geq 2p-1$, then
\[
 \|(T-T_h)_{|E}\|_{\mc{L}(V(h))}\leq C \Lr{h^{\tau}+\Lam_md^{\tau}},
\]
where $\tau=m+1-p$.
\end{lemma}

\begin{proof}
The source problem corresponding to~\eqref{eq:elliptic} takes the form
\[
- \Delta u_s = f  \quad \text{in }\Omega,\quad 
u_s =0 \quad \text{on }\partial \Omega,
\]
and the source problem corresponding to~\eqref{eq:harmonic} takes the
form
\[
\Delta^2 u_s = f \quad \text{in } \Omega,\quad u_s=\dfrac{\partial
u_s}{\partial \un}=0 \quad  \text{on } \partial \Omega.
\]

The discrete variational problem for the
source problem reads: find $u_h\in U_h$ such that 
\begin{equation}
  a_h(\mc{R} u_h, \mc{R} v_h) = (f, \mc{R} v_h), \quad \forall v_h \in
  U_h.
  \label{eq:dsource}
\end{equation}
From~\cite[Theorem 3.1]{li2016discontinuous} \cite[Theorem
3.1]{li2017discontinuous}, we conclude that there exists a unique
solution to \eqref{eq:dsource}. Furthermore, if $u_s \in
H^{m+1}(\Omega)$, we have the following estimate:
\[
\|u_s-\mc{R}u_h\|_{V(h)} \leq C (h^{\tau}+\Lam_m
d^{\tau})|u_s|_{H^{m+1}(\Omega)},
\]
where $\tau=m+1-p$. This estimate directly implies 
\[
 \|(T-T_h)_{|E}\|_{\mc{L}(V(h))}\leq C \Lr{h^{\tau}+\Lam_md^{\tau}},
\]
which completes the proof.
\end{proof}

Then, the error estimates for the eigenfunctions can be directly derived.

\begin{theorem}\label{approx_eigenvector}
Let $u^{(i)}$ be a unit
eigenfunction associated with an eigenvalue $\lambda^{(i)}$ of
multiplicity $k$, such that $\lambda^{(i)}=\cdots=\lambda^{(i+k-1)}$,
and $\mc{R}u^{(i)}_h,\cdots, \mc{R}u^{(i+k-1)}_h$ denote the discrete
eigenfunctions associated with the $k$ discrete eigenvalues converging
to $\lambda^{(i)}$. Then there exists
\begin{equation}\label{eq:discrete_eigenspace}
\mc{R}w^{(i)}_h\in \text{span} \{\mc{R}u^{(i)}_h,\cdots,
\mc{R}u^{(i+k-1)}_h\},
\end{equation}
such that
\begin{equation}\label{eq:best_appr_eigenspace}
\|u^{(i)}-\mc{R}w^{(i)}_h\|_{V(h)}\leq C \sup_{u\in E,\|u\|_{V(h)}=1}
\inf_{v\in U_h} \|u-\mc{R}v\|_{V(h)}.
\end{equation}
Moreover, if the regularity of eigenspace is $E\subset
H^{m+1}(\Omega)$, $m\geq 2p-1$, then
\begin{equation}\label{eq:appr_eigenspace}
\|u^{(i)}-\mc{R}w^{(i)}_h\|_{V(h)}\leq C \Lr{h^{\tau}+\Lam_md^{\tau}}
|u^{(i)}|_{H^{m+1}(\Omega)},
\end{equation}
where $\tau=m+1-p$.
\end{theorem}

\begin{proof}
The results~\eqref{eq:discrete_eigenspace}
and~\eqref{eq:best_appr_eigenspace} are direct extensions of
Lemma~\ref{corollary_non-pollution} and the estimate
~\eqref{eq:appr_eigenspace} is directly derived from
Lemma~\ref{lemma:appro_eigenvector}.
\end{proof}

Finally, the error estimates for the eigenvalues of
\eqref{eq:elliptic} and \eqref{eq:harmonic} are the following.
\begin{theorem}\label{approx_eigenvalue} 
Let $\lambda^{(i)}$ denote the
eigenvalue of \eqref{eq:elliptic} and \eqref{eq:harmonic} with
multiplicity $k$, $\lambda^{(i)}_h$ be the discrete eigenvalues and $E$
denote the eigenspace associated with $\lambda^{(i)}$, then we have
\begin{equation}\label{eq:best_appro_eigenvalue}
|\lambda^{(i)}-\lambda^{(i)}_h|\leq C \sup_{u\in E,\|u\|_{V(h)}=1}
\inf_{v\in U_h} \|u-\mc{R}v\|_{V(h)}^2.
\end{equation}
Moreover, if eigenspace $E\subset H^{m+1}(\Omega)$, $m\geq 2p-1$, then
the following optimal double order of convergence holds
\begin{equation}\label{eq:eigenvalue_estimate}
|\lambda^{(i)}-\lambda^{(i)}_h|\leq C \Lr{h^{2\tau}+\Lam_md^{2\tau}},
\end{equation}
where $\tau = m + 1 -p$.
\end{theorem}

\begin{proof}
Since the operator $a_h(\cdot, \cdot)$ is symmetric, i.e. $a_h(T_h f,
\mc{R}v) = a_h(\mc{R}v, T_h f)$, the estimate
\eqref{eq:best_appro_eigenvalue} is a direct application of
~\cite[Theorem 9.13]{boffi2010finite}. The
estimate~\eqref{eq:eigenvalue_estimate} is the combination
of the inequality~\eqref{eq:best_appro_eigenvalue} and
Lemma~\ref{lemma:appro_eigenvector}.
\end{proof}


\section{NumericalResults}\label{sec:examples}
In this section, we present some
numerical results to show that our method is efficient for eigenvalue
problems if we use higher order approximation. We would like to
emphasize two points: \begin{itemize} \item Less DOFs are used by our
method for high order approximation   comparing to the classical DG
method and conforming finite element   methods; \item More reliable
eigenvalues can be obtained increasing the order   of approximation.
\end{itemize} Besides, we will compute the numerical order of
convergence to verify the theoretical error estimates and give results
on different domains and different meshes to demonstrate the
flexibility of the implementation using our method.

\subsection{Examples setup} First, let us list the setup of the
examples to be investigated.

\paragraph{\bf Example 1} We consider the two-dimensional square
domain $\Omega=[0,\pi]^2$, the eigenpairs of problem
\eqref{eq:elliptic} are given by \begin{displaymath}   \begin{aligned}
\lambda_{i,j}=& i^2+j^2, \ \text{for} \ i,j>0 \ \text{and} \ i,j
\in \mb{N},\\ u_{i,j}=& \sin(ix)\sin(jy),   \end{aligned}
\end{displaymath} and for the problem \eqref{eq:harmonic} with the
boundary condition $u|_{\partial \Omega}=\Delta u|_{\partial
\Omega}=0$, which is related to the bending of a simply supported
plate \cite{brenner2015C0}, the eigenpairs are given by
\begin{displaymath}   \begin{aligned}     \lambda_{i,j}=& (i^2+j^2)^2,
\ \text{for} \ i,j>0 \ \text{and}     \ i,j \in \mb{N},\\ u_{i,j}=&
\sin(ix)\sin(jy).   \end{aligned} \end{displaymath} In this example,
the computation involves a series of regular unstructured triangular
meshes which are generated by {\it   Gmsh}~\cite{geuzaine2009gmsh}.
For the second order elliptic problem we take $(m=1,2,3,4,5)$ and for
the biharmonic problem $m$ is taken as $( 2,3,4,5)$.

\paragraph{\bf Example 2} We consider the L-shaped domain $[-1,1]^2
\backslash (0,1]\times(0,-1]$. The domain is partitioned     into
polygonal meshes by {\it PolyMesher}
\cite{talischi2012polymesher}. Figure \ref{polygonal_mesh} shows
the initial mesh and the refined mesh.  The meshes contain the
elements with various geometries such as quadrilaterals,
pentagons, hexagons, and so on. The first eigenfunction in
L-sharped domain has a singularity at the reentrant corner and has
no analytical expression. We note that the third eigenpair is
smooth for L-shaped domain. For the second order elliptic
equation, the third eigenvalue is $2\pi^2$ and the corresponding
eigenfunction is $\sin(\pi x) \sin(\pi y)$, and we take
$(m=1,2,3)$ to solve the eigenvalue problem.  For the biharmonic
equation, the third eigenpair is $4\pi^4$ and $\sin(\pi x)
\sin(\pi y)$, and we choose $(m=2,3)$ to solve it.

\begin{figure}   \begin{center}
\includegraphics[width=0.48\textwidth]{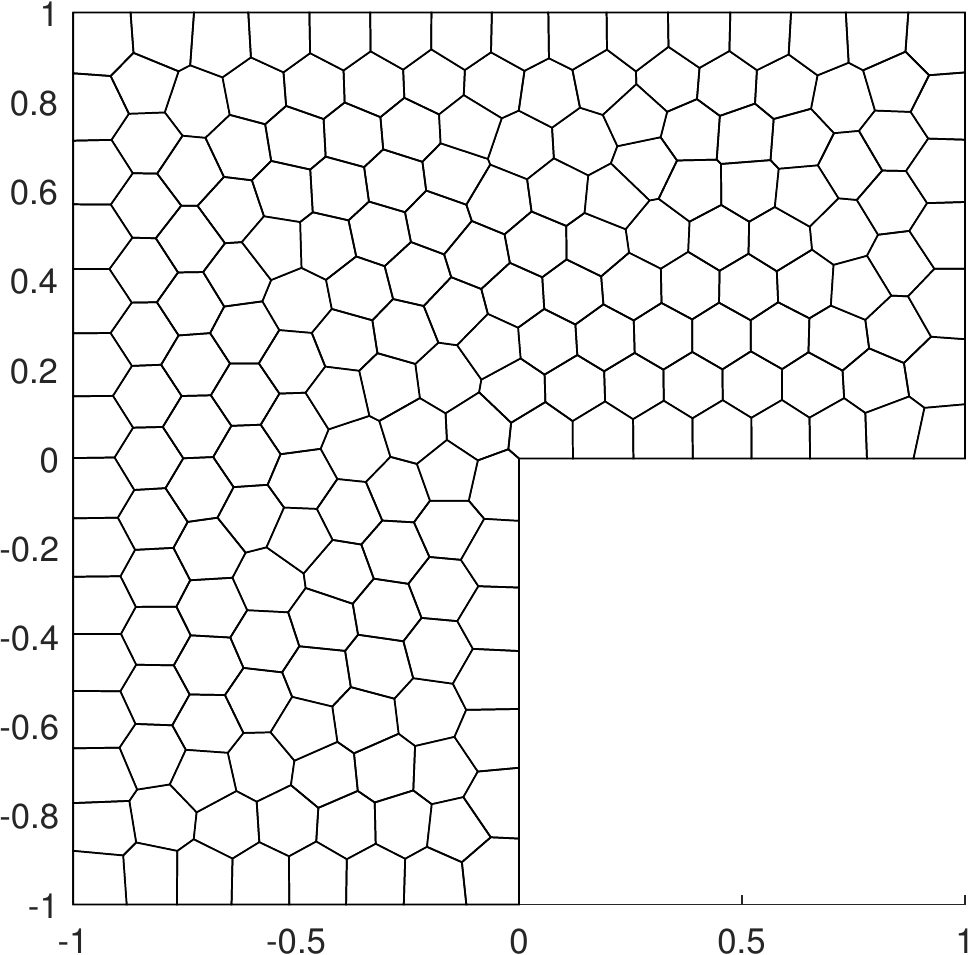}
\includegraphics[width=0.48\textwidth]{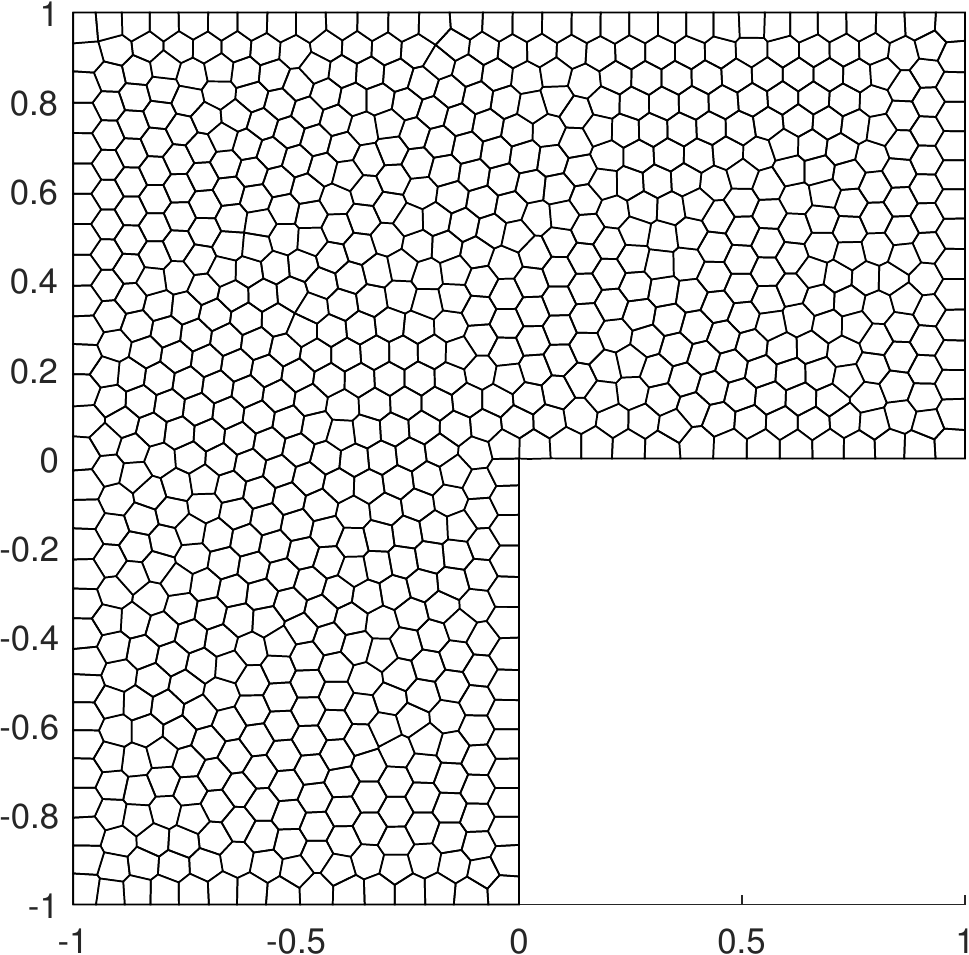}
\caption{The polygonal mesh (left) / refined polygonal mesh
(right) for Example 2.}     \label{polygonal_mesh}   \end{center}
\end{figure}

\paragraph{\bf Example 3} We solve the eigenvalue problem in three
dimensions in this example. The computational domain is the unit cubic
$\Omega=[0,1]^3$ which is partitioned into tetrahedral meshes by {\it
Gmsh}. The eigenpairs of problem \eqref{eq:elliptic} are as follows:
\begin{displaymath}   \begin{aligned}     \lambda_{i,j,k}=&
(i^2+j^2+k^2)\pi^2, \ \text{for} \ i,j,k>0     \ \text{and} \ i,j,k
\in \mb{N},\\ u_{i,j,k}=& \sin(i\pi     x)\sin(j\pi y)\sin(k\pi z),
\end{aligned} \end{displaymath} and for problem \eqref{eq:harmonic}
with the simply supported plate boundary condition, the eigenpairs are
given by \begin{displaymath}   \begin{aligned}     \lambda_{i,j,k}=&
(i^2+j^2+k^2)^2\pi^4, \ \text{for} \ i,j,k>0     \ \text{and} \ i,j,k
\in \mb{N},\\ u_{i,j,k}=& \sin(i\pi     x)\sin(j\pi y)\sin(k\pi z).
\end{aligned} \end{displaymath}

\subsection{Convergence order study} At first, we show that the
numerical results verify the optimal convergence order as predicted by
the theory.

For Example 1, Figure \ref{tri_elliptic_error} shows the convergence
rates of the eigenvalue and eigenfunction. The exact $20$-th
eigenvalue is $32$ and the corresponding eigenfunction is
$\sin(4x)\sin(4y)$. The eigenvalue converges to the exact one with
$h^{2m}$ rate and for the eigenfunction the convergence rate is $h^m$.
Figure \ref{tri_biharmonic_error} shows the convergence rates of the
eigenvalue and eigenfunction of the biharmonic equation. The exact
$20$-th eigenvalue is $1024$ and the corresponding eigenfunction is
$\sin(4x)\sin(4y)$. The eigenvalue convergences to the exact one with
$h^{2(m-1)}$ rate, and for the eigenfunction the convergence rate is
$h^{m-1}$. The numerical results agree with Theorem
\ref{approx_eigenvector} and \ref{approx_eigenvalue} perfectly.

\begin{figure}   \begin{center}
\includegraphics[width=0.48\textwidth]{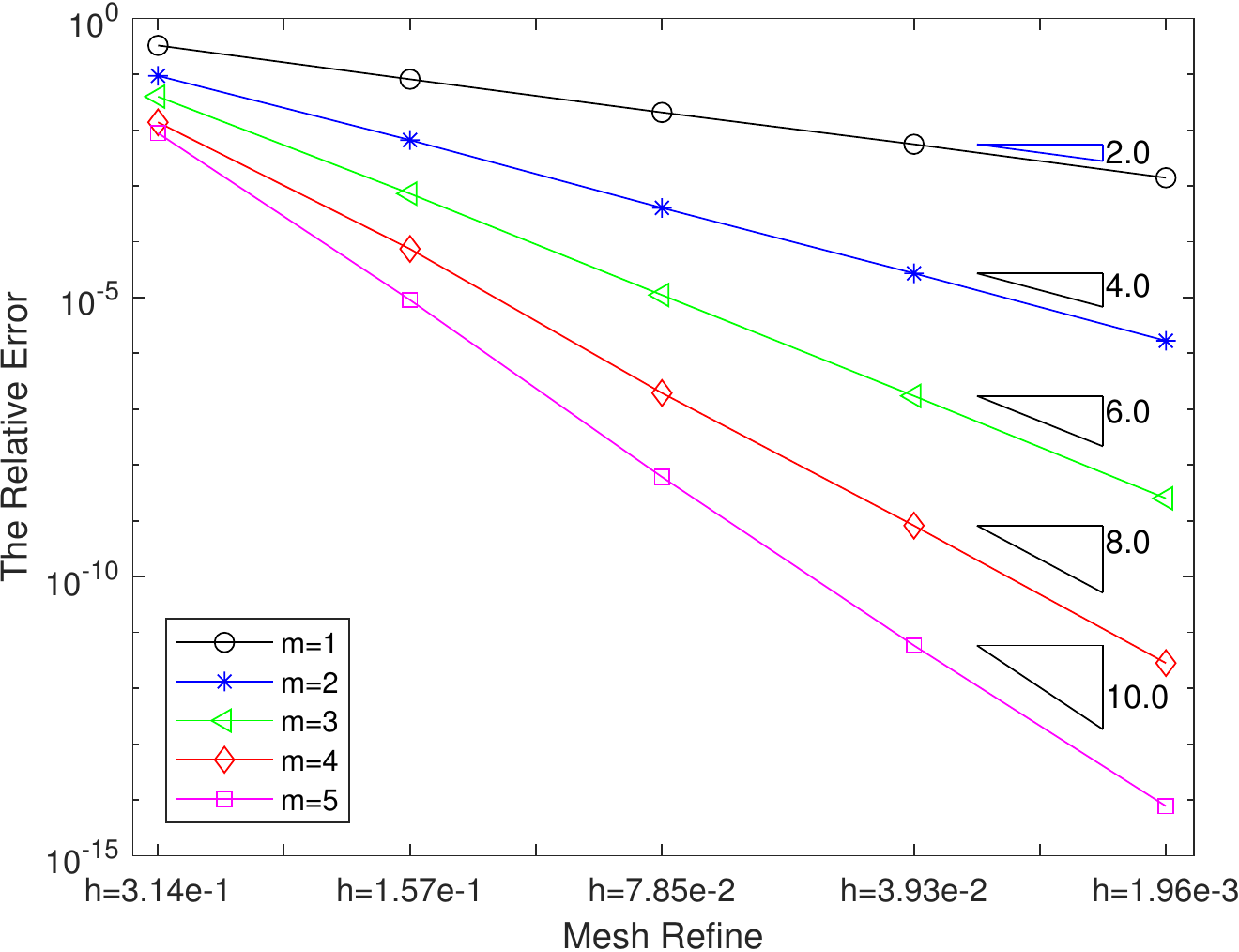}
\includegraphics[width=0.48\textwidth]{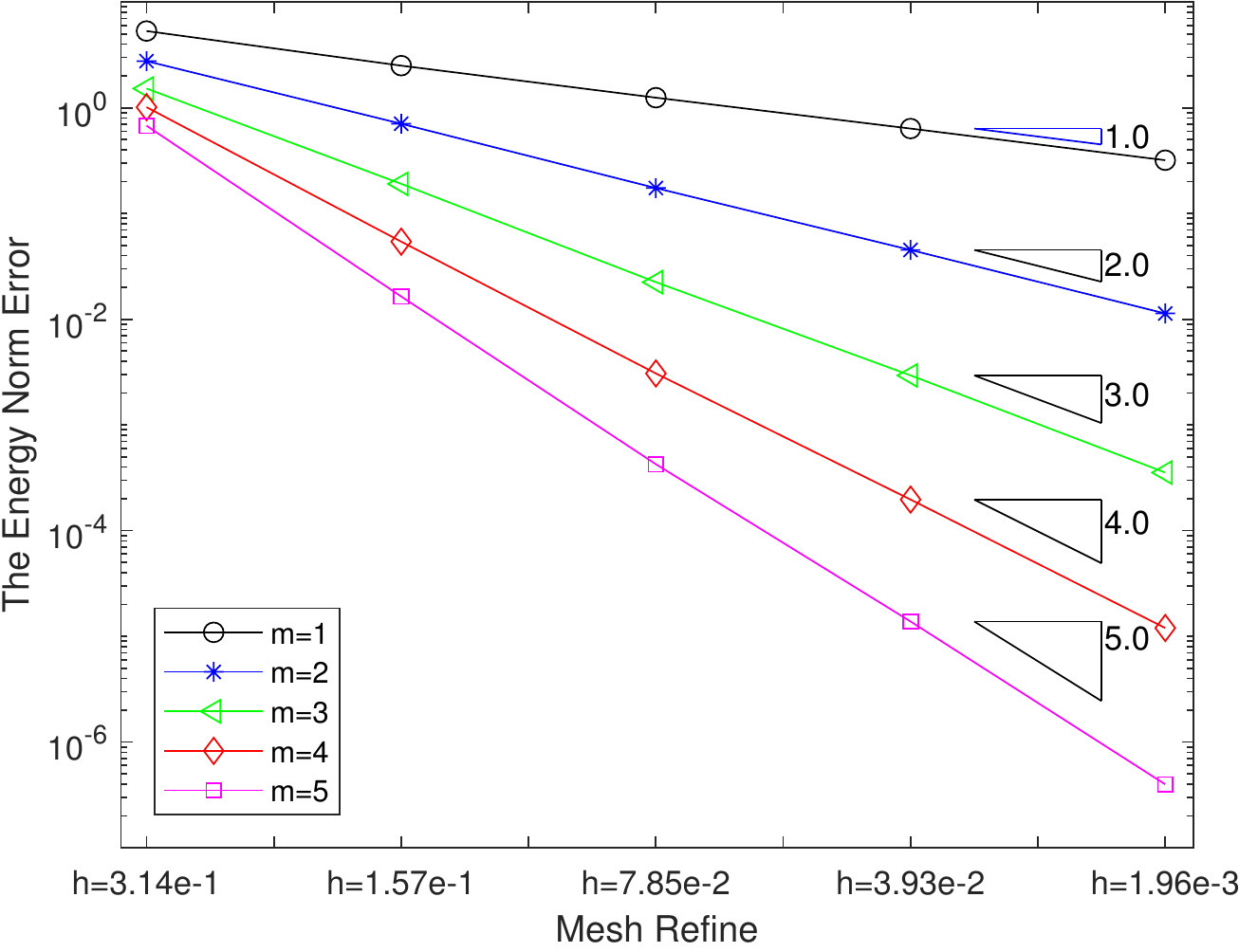}
\caption{The convergence rates of the $20$-th eigenvalue (left) /       eigenfunction (right)
of the second order problem for different       orders $m$ on triangle
meshes for Example 1.}
\label{tri_elliptic_error}   
\end{center}
\end{figure}

\begin{figure}   \begin{center}
\includegraphics[width=0.48\textwidth]{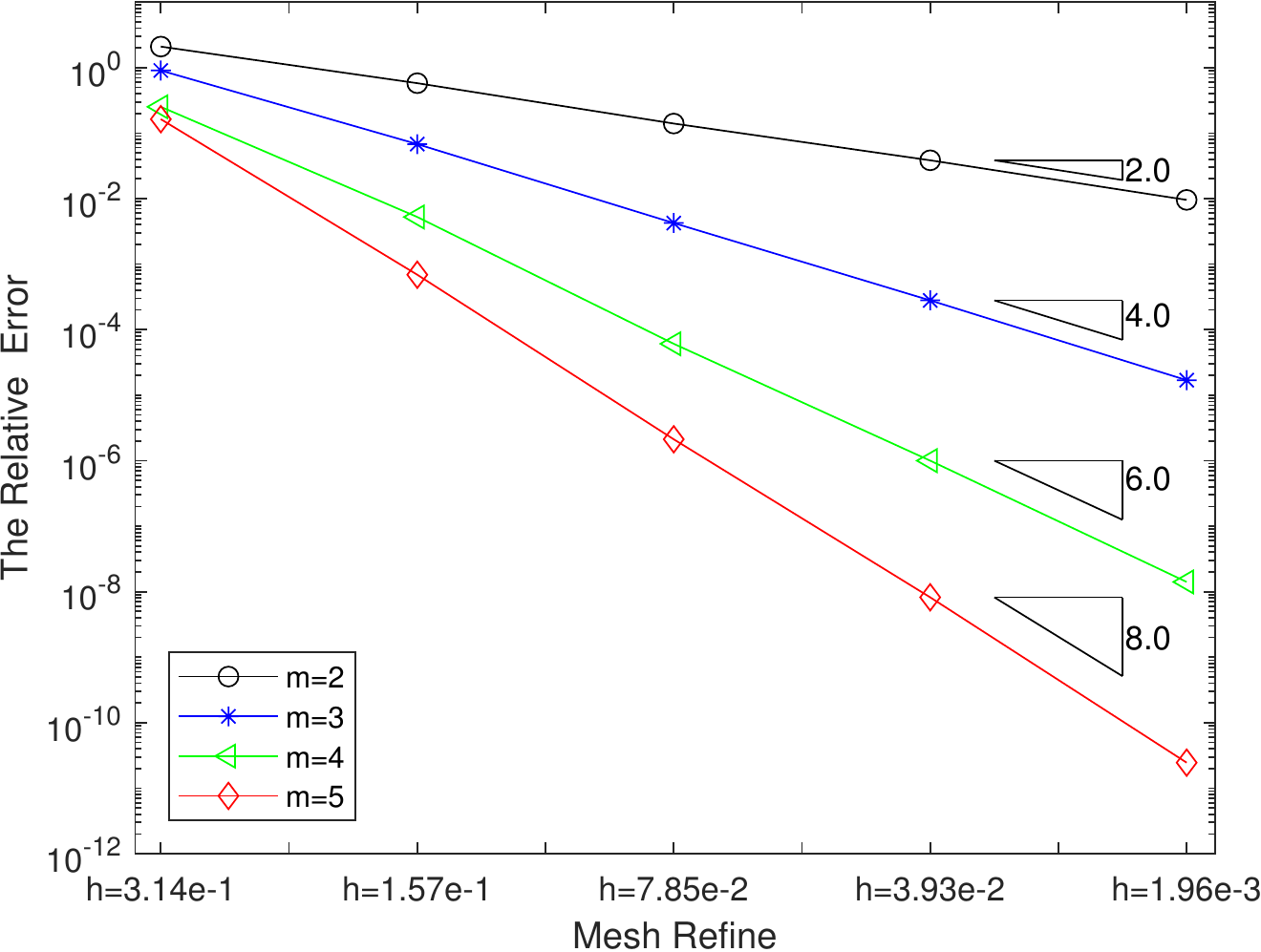}
\includegraphics[width=0.48\textwidth]{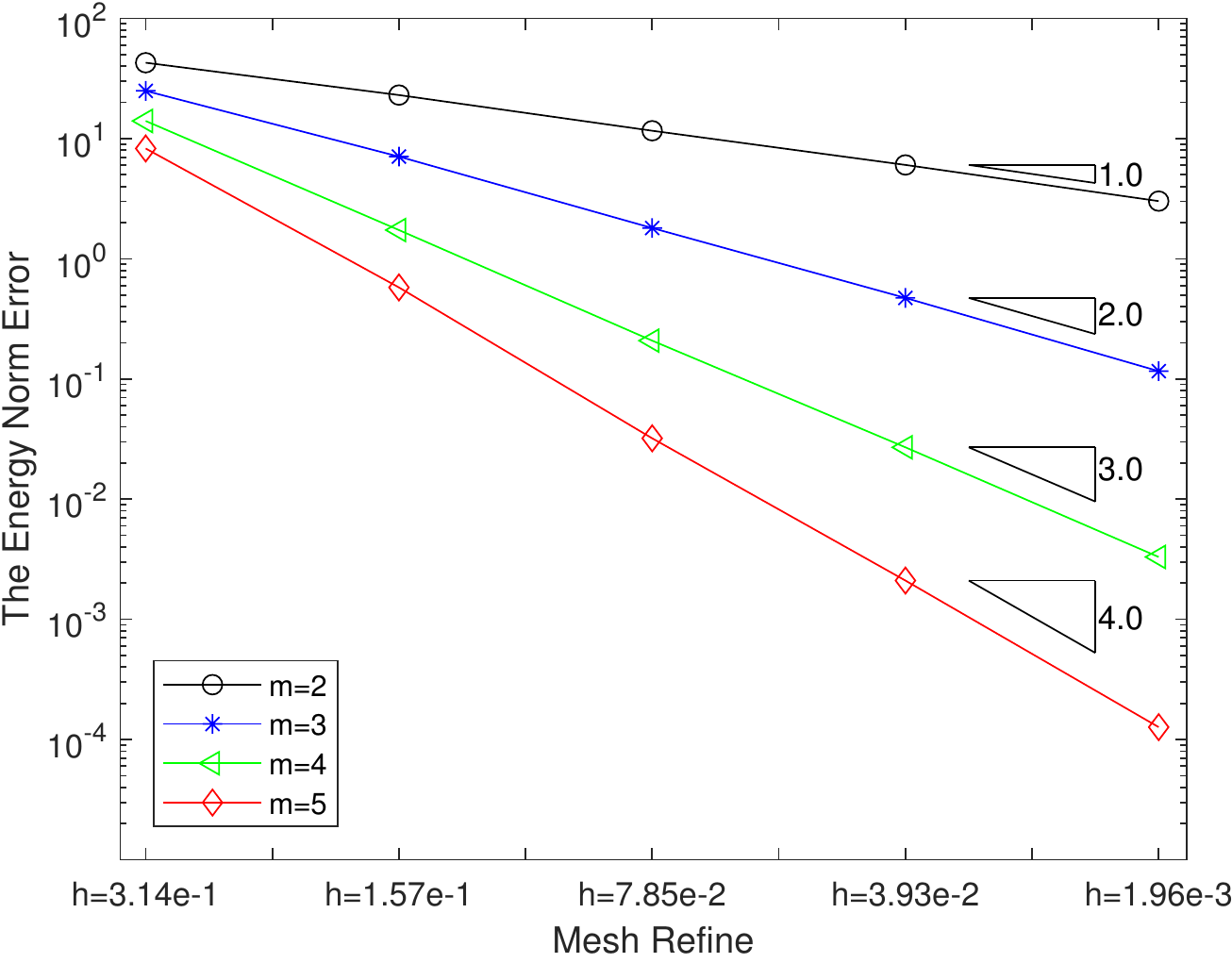}
\caption{The convergence rates of the $20$-th eigenvalue (left) /       eigenfunction (right)
of the biharmonic problem for different       orders $m$ on triangle
meshes for Example 1.}     
\label{tri_biharmonic_error}   
\end{center}
\end{figure}

The Example 2 shows that the proposed method can handle these
polygonal elements easily. First, we calculate the third smooth
eigenpair to verify the analysis of the proposed method. Figures
\ref{polygonal_elliptic_error} and \ref{polygonal_biharmonic_error}
show the numerical results that agree with the theoretical prediction.
The values of the first eigenvalue of the Laplace/biharmonic equation
are shown in Table \ref{table_L_shape}. It is clear that the
eigenvalues converge to the real eigenvalue as $h$ approaches $0$. The
eigenfunctions corresponding to the first eigenvalue and third
eigenvalue are presented in Figures \ref{polygonal_elliptic_eigfun}
and \ref{polygonal_biharmonic_eigfun}.

\begin{figure}   \begin{center}
\includegraphics[width=0.48\textwidth]{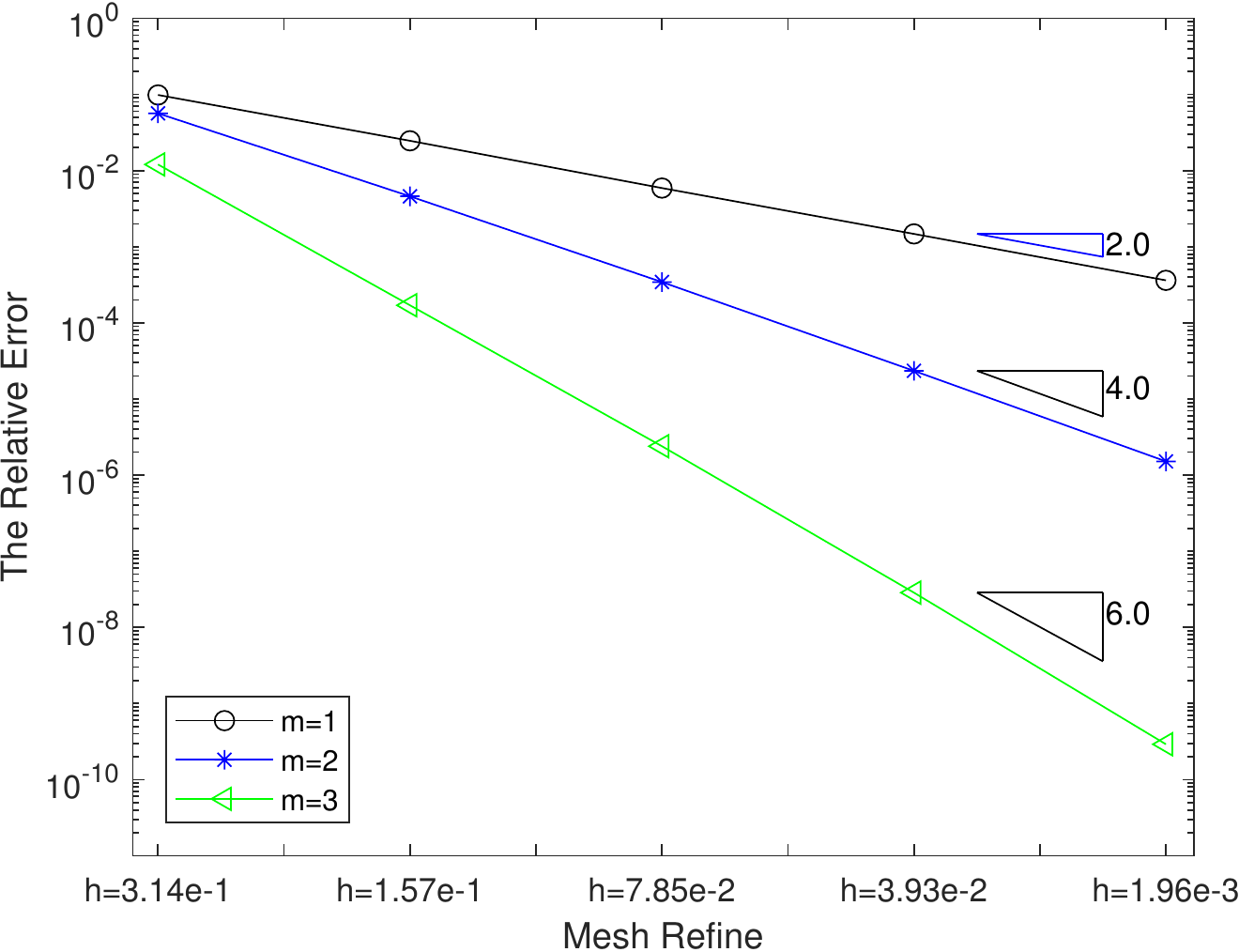}
\includegraphics[width=0.48\textwidth]{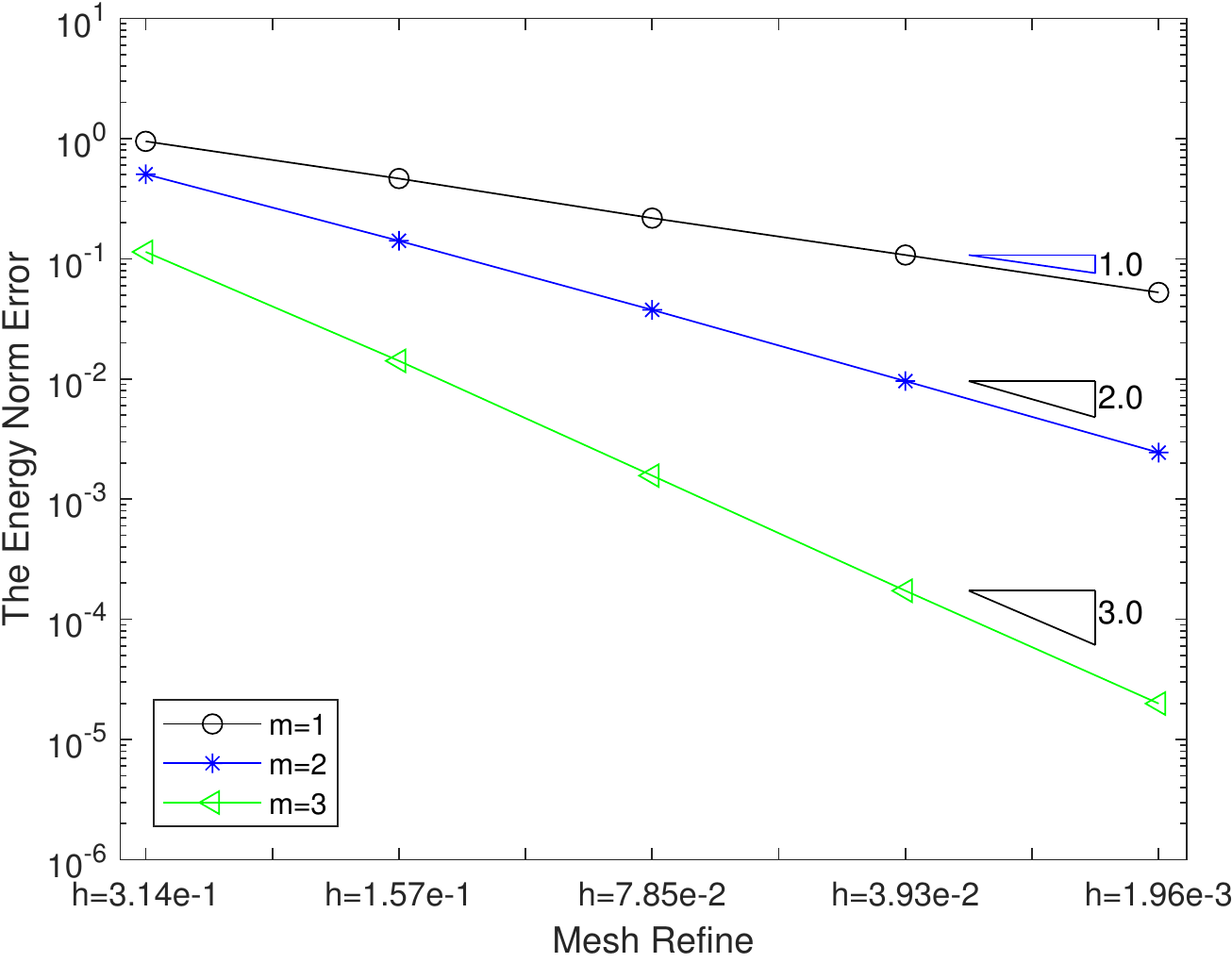}
\caption{The convergence rates of the 3rd eigenvalue (left) /
eigenfunction (right) of the second order problem for different
orders $m$ on polygonal meshes for Example 2.}
\label{polygonal_elliptic_error}   
\end{center} 
\end{figure}

\begin{figure}   \begin{center}
\includegraphics[width=0.48\textwidth]{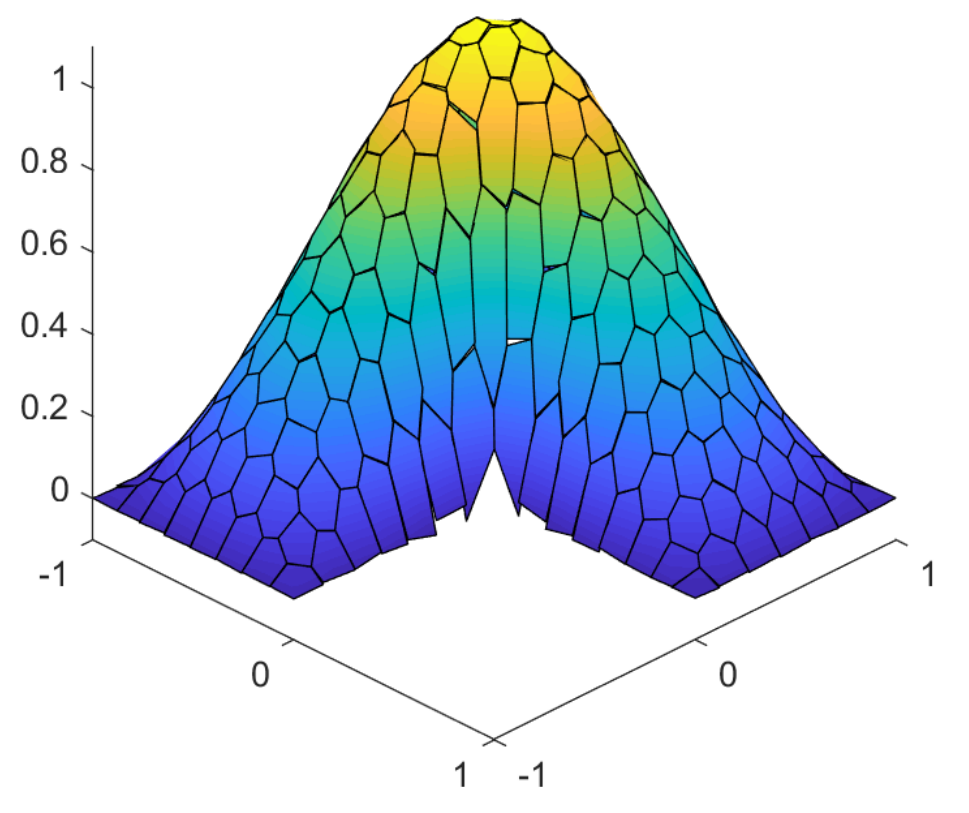}
\includegraphics[width=0.48\textwidth]{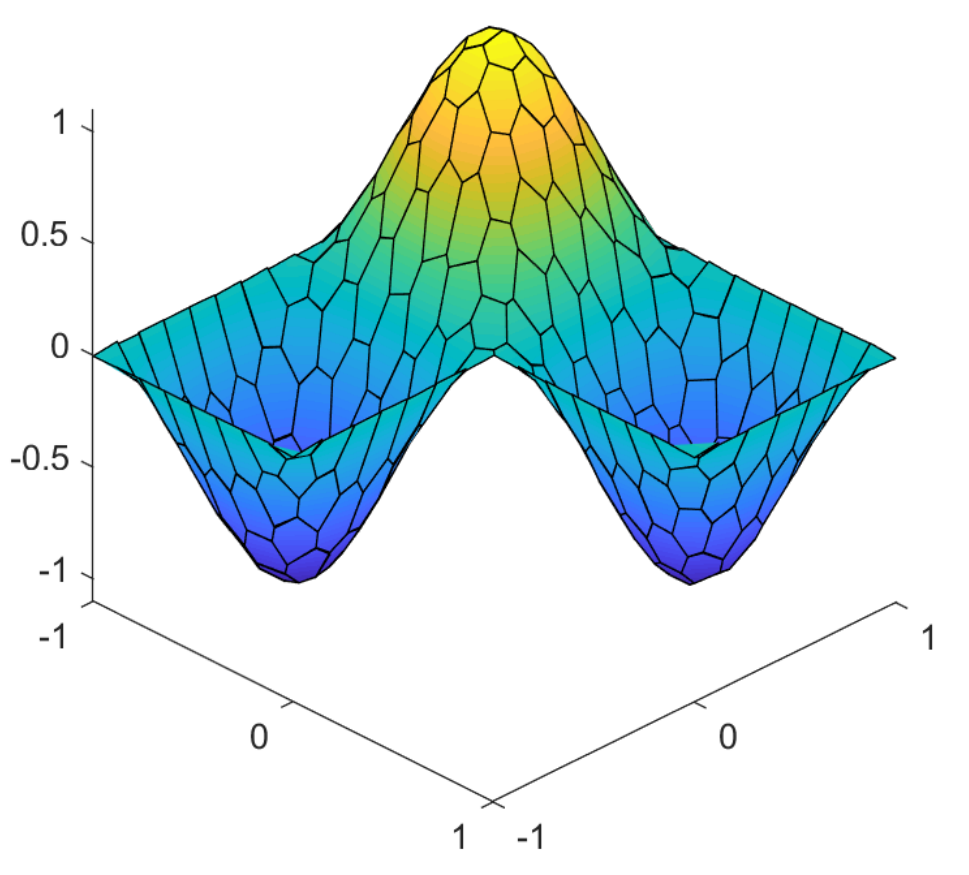}     
\caption{The 1st eigenfunction(left) and the 3rd       eigenfunction(right) of the
second order problem for Example 2.}
\label{polygonal_elliptic_eigfun}   
\end{center} 
\end{figure}

\begin{figure}   \begin{center}
\includegraphics[width=0.48\textwidth]{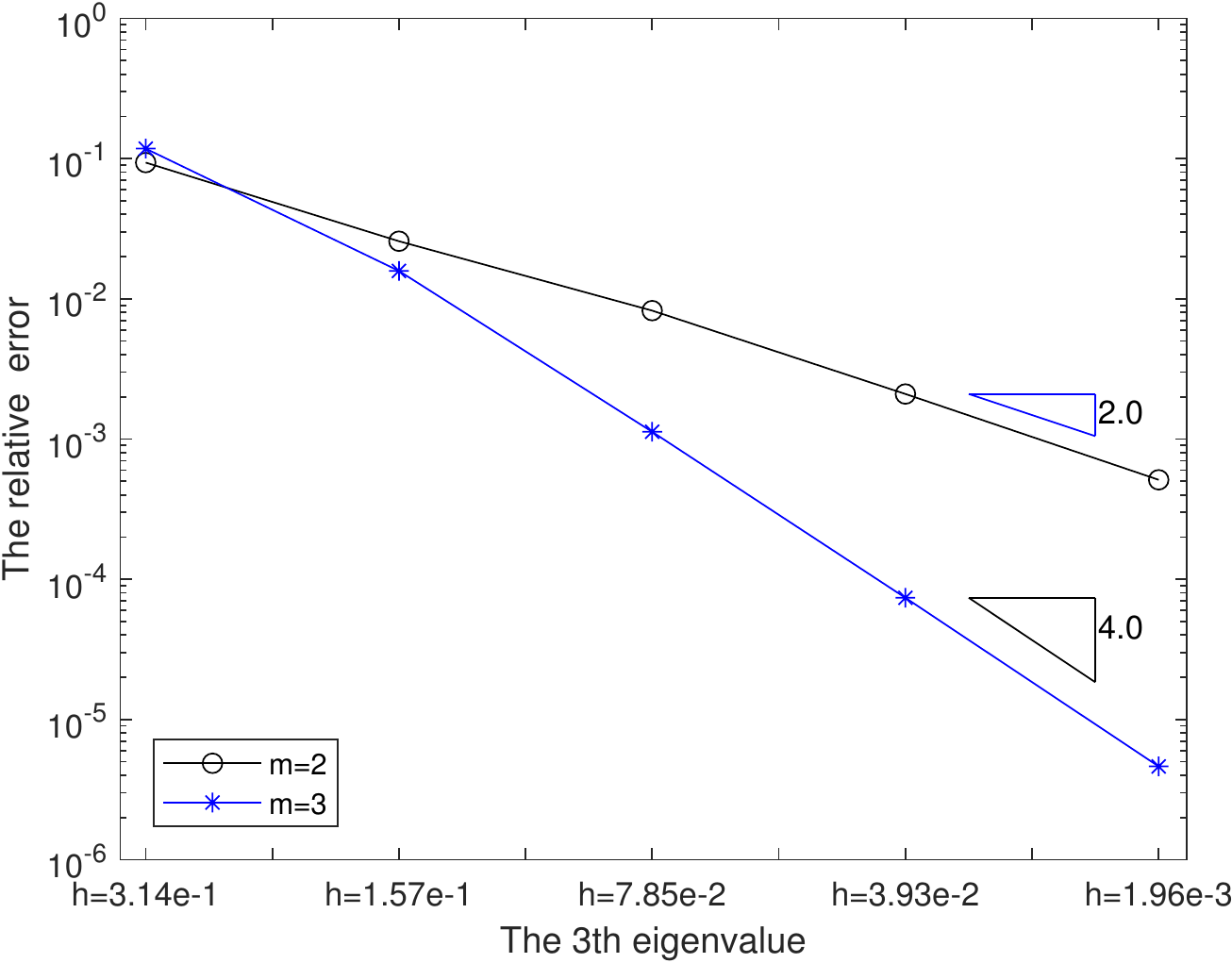}
\includegraphics[width=0.48\textwidth]{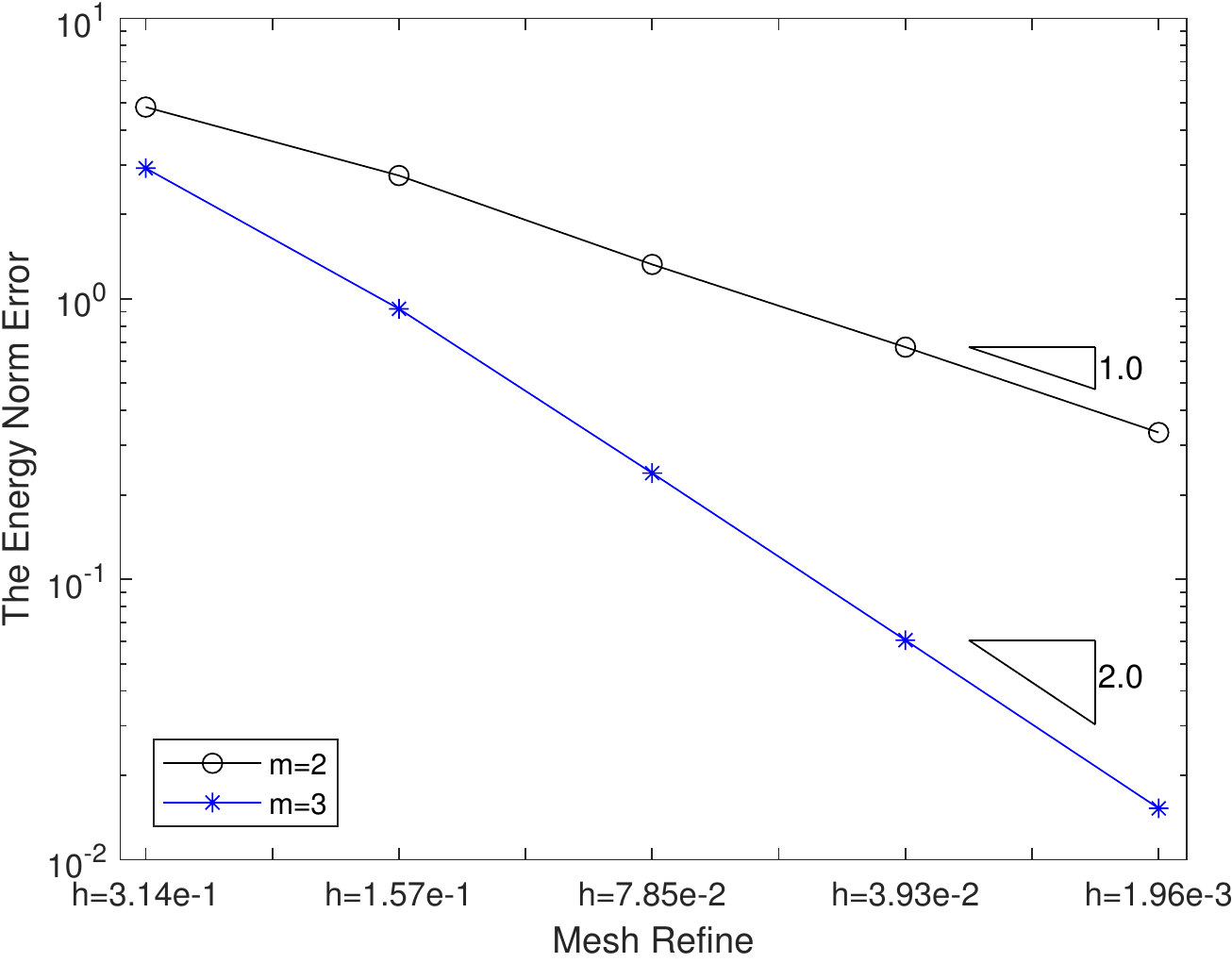}
\caption{The convergence rates of 3rd eigenvalue (left) /
eigenfunction (right) of the biharmonic problem for different
orders $m$ on polygonal meshes for Example 2.}
\label{polygonal_biharmonic_error}   
\end{center} 
\end{figure}

\begin{figure}   
\begin{center}
\includegraphics[width=0.48\textwidth]{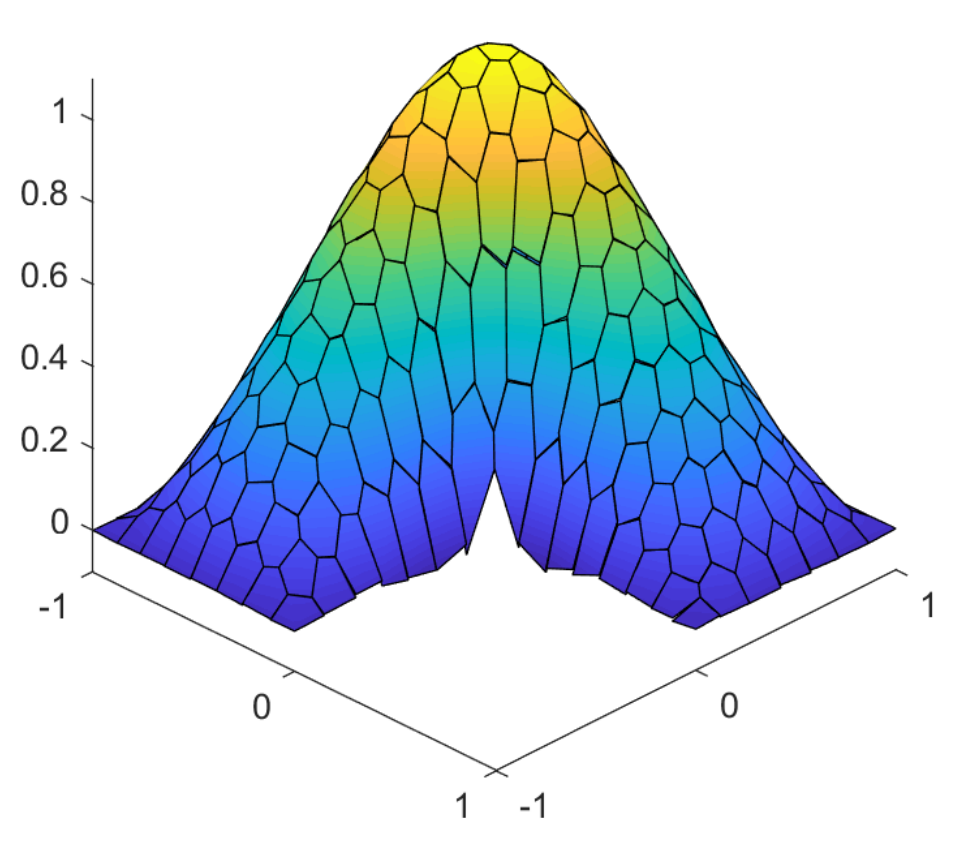}     
\includegraphics[width=0.48\textwidth]{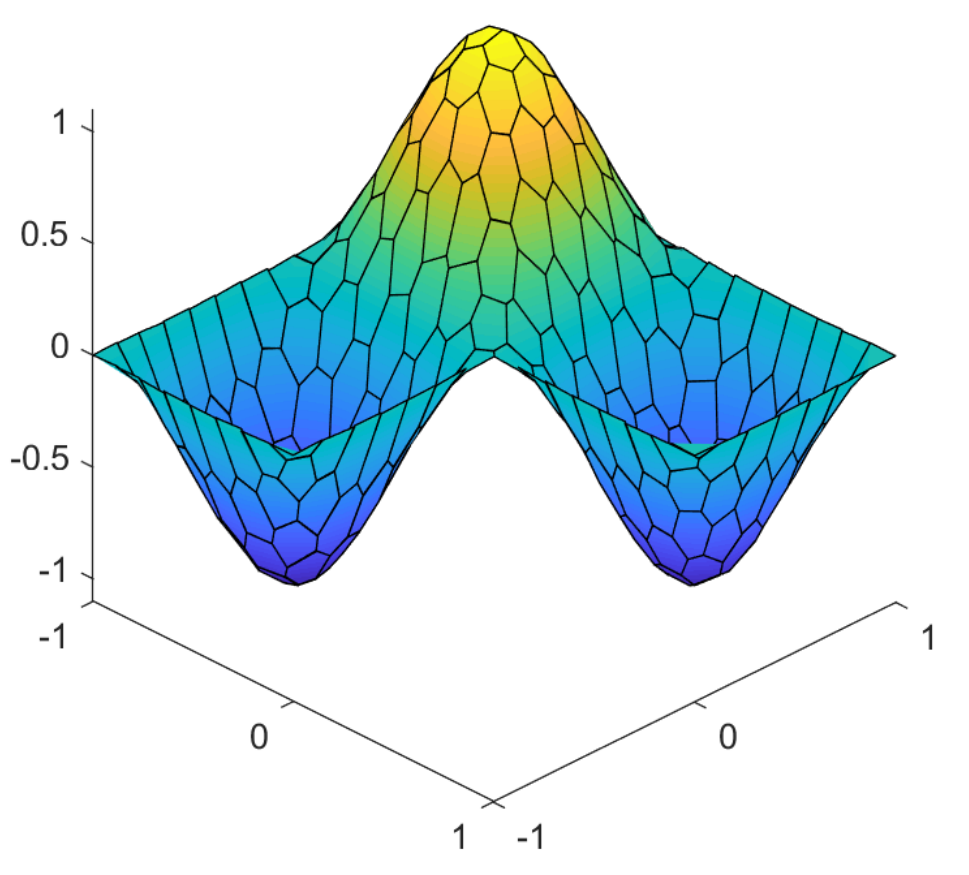}     
\caption{The 1st eigenfunction (left) and the 3rd eigenfunction (right) of the
biharmonic problem for Example 2.}
\label{polygonal_biharmonic_eigfun}   
\end{center} 
\end{figure}

\begin{table} 
\begin{center} 
\scalebox{0.9}{
\begin{tabular}{|c|c|c|c|c|c|c|} \hline Order & DOFs & N=2.00E+2 &
N=8.00E+2 & N=3.20E+3 & N=1.28E+4& N=5.12E+4 \\ \hline $m=1$ &
\multirow{3}{*}{Laplace} &10.786 & 9.9562 & 9.7396 & 9.6867 & 9.6733
\\ \cline{1-1} \cline{3-7} $m=2$ & &10.403 & 9.7422 & 9.6780& 9.6707 &
9.6692 \\ \cline{1-1} \cline{3-7} $m=3$ & &9.8128 & 9.6811 & 9.6724 &
9.6700 & 9.6691 \\ \hline \hline $m=2$ & \multirow{2}{*}{Biharmonic}
&179.98 & 171.55 & 167.78 & 166.75 & 165.43 \\ \cline{1-1} \cline{3-7}
$m=3$ & &168.82 & 166.78 & 165.77 & 165.12 & 164.68 \\ \hline
\end{tabular}} 
\caption{The first eigenvalues of the second order and
biharmonic equation in L-shaped domain.} \label{table_L_shape}
\end{center} 
\end{table}

For Example 3, the numerical results are presented in Tables
~\ref{table_elliptic_3D} and ~\ref{table_biharmonic_3D} for the second
order and biharmonic equation, respectively. The convergence order of
the second order equation is $h^{2m}$, and of the biharmonic equation
is $h^{2(m-1)}$.  Obviously, the computational results agree with the
error estimates.

\begin{table} \centering \begin{tabular}{|c|c|c|c|c|c|}   \hline Order
& Mesh Size & $h=$2.500E-1 & $h=$1.250E-1 &   $h=$6.250E-2 &
$h=$3.125E-3 \\ \hline \multirow{3}{*}{$m=1$} & Value   & 45.43 &
34.99 & 31.03 & 29.97 \\ \cline{2-6} & Error & 5.33E-1 &   1.81E-1 &
4.82E-2 & 1.24E-2 \\ \cline{2-6} & Order & - & 1.55 & 1.91   & 1.96 \\
\hline \multirow{3}{*}{$m=2$} & Value & 35.57 & 29.96 &   29.63 &
29.61 \\ \cline{2-6} & Error & 2.01E-1 & 1.21E-2 & 7.70E-4 &   4.98E-5
\\ \cline{2-6} & Order & - & 4.10 & 3.96 & 3.95 \\ \hline
\multirow{3}{*}{$m=3$} & Value & 31.34 & 29.63 & 29.61 & 29.61   \\
\cline{2-6} & Error & 5.85E-2 & 6.79E-4 & 1.07E-5 & 1.64E-7   \\
\cline{2-6} & Order & - & 6.42 & 5.99 & 6.02 \\ \hline
\multirow{3}{*}{$m=4$} & Value & 30.23 & 29.61 & 29.61 & 29.61   \\
\cline{2-6} & Error & 2.12E-2 & 8.24E-5 & 3.23E-7 & 1.23E-9   \\
\cline{2-6} & Order & - & 8.03 & 7.99 & 7.93 \\ \hline \end{tabular}
\caption{The first eigenvalues of the Laplace problem in 3D,
$\lambda_{1}=3\pi^2(29.61)$.} \label{table_elliptic_3D} \end{table}

\begin{table} \centering \begin{tabular}{|c|c|c|c|c|c|}   \hline Order
& Mesh Size & h=2.500E-1 & h=1.250E-1 & h=6.250E-2 &   h=3.125E-3 \\
\hline \multirow{3}{*}{m=2} & Value & 1000.54 & 906.41   & 883.20 &
878.20 \\ \cline{2-6} & Error & 1.41E-1 & 3.39E-2 &   7.44E-3 &
1.73E-3 \\ \cline{2-6} & Order & - & 2.05 & 2.18 & 2.09   \\ \hline
\multirow{3}{*}{m=3} & Value & 942.74 & 879.49 & 876.84 &   876.69 \\
\cline{2-6} & Error & 7.54E-2 & 3.21E-3 & 1.88E-4 &   1.07E-5 \\
\cline{2-6} & Order & - & 4.55 & 4.09 & 4.12 \\ \hline
\multirow{3}{*}{m=4} & Value & 897.55 & 876.85 & 876.68 & 876.68   \\
\cline{2-6} & Error & 2.38E-2 & 2.00E-4 & 2.91E-6 & 4.33E-8   \\
\cline{2-6} & Order & - & 6.89 & 6.10 & 6.07 \\ \hline \end{tabular}
\caption{The first eigenvalues of the biharmonic problem in 3D,
$\lambda_{1}=9\pi^4(876.68)$.} \label{table_biharmonic_3D} \end{table}

\begin{remark}   We note that all the eigenvalues obtained by the
proposed method are   greater than the exact eigenvalues. This
behavior appears if   conforming finite element method is used to
solve the eigenvalue   problem. However, the approximate space $V_h$
is not a subspace of   the space $V=H^{1}_{0}$ or $H^{2}_{0}$. In DG
framework, this phenomenon is related to the penalty parameter.
Warburton and Embree studied the role of penalty in the LDG method for
Maxwell's eigenvalue problem in~\cite{warburton2006role}.  Giani et
al.~\cite{Giani2018posteriori} used the asymptotic perturbation theory
to analyze the dependence of eigenvalues and eigenspaces on the
penalty parameter.    We hope the reason why this   happened in our
method can be clarified in future study. \end{remark}

\subsection{Efficiency in terms of number of DOFs } Next, we make a
comparison in terms of number of DOFs among different methods. For the
second order elliptic problem, we consider the conforming FEM,
standard SIPDG method~ \cite{antonietti2006discontinuous} and our
method. For the biharmonic problem, we consider the $C^0$ IPG,
standard SIPDG and our method. Here we will study the numerical
behavior for higher order approximation. We restrict to Example 1,
since in this case the solution has enough regularity.

We calculate the first eigenvalue and eigenfunction on successively
refined meshes. The errors of eigenvalue are measured in the relative
error, and the errors of the eigenfunction are measured in
$|\cdot|_{1,h}$ and $|\cdot|_{2,h}$ semi-norms, respectively.

For the Laplace problem, Figure \ref{tri_elliptic_compare} shows the
performance of the conforming FEM, SIPDG method and our method. The
approximation order $m$ is taken from $1$ to $4$.  The convergence
rate for the eigenvalue is $h^{2m}$ and for the eigenfunction the rate
is $h^{m}$ which meet the theoretical predictions. The horizontal
ordinate is the number of DOFs. The number of DOFs employed by our
method is fixed while the approximation order increases. In all cases,
the SIPDG method uses the maximum number of DOFs.  As one's
expectation, the figure shows that the efficiency of FEM is higher
than others for the low order approximation. Increasing of the
approximation order, our method becomes the most efficient method
among these three methods.

\begin{figure}   \begin{center}
\includegraphics[width=0.48\textwidth
]{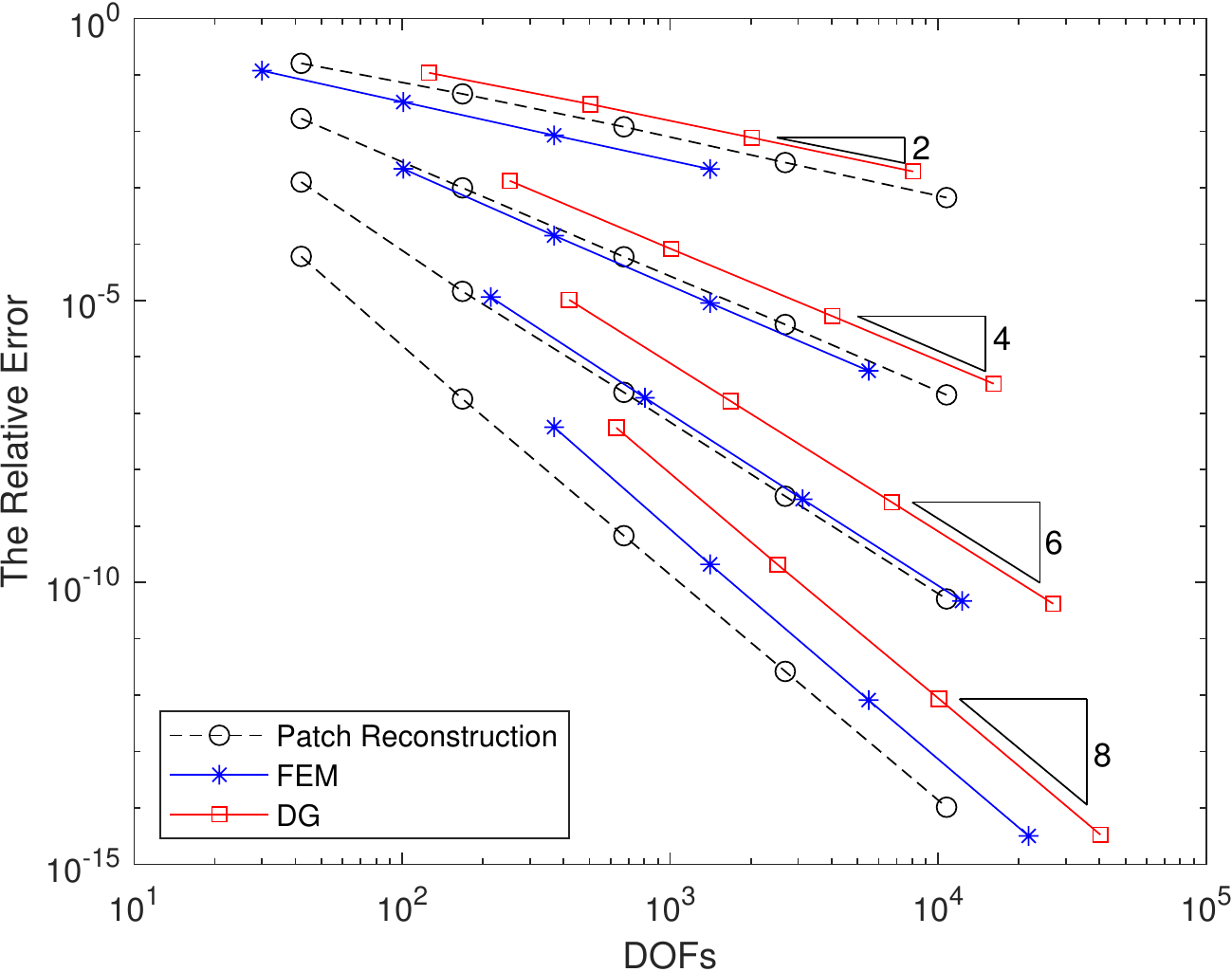}
\includegraphics[width=0.48\textwidth
]{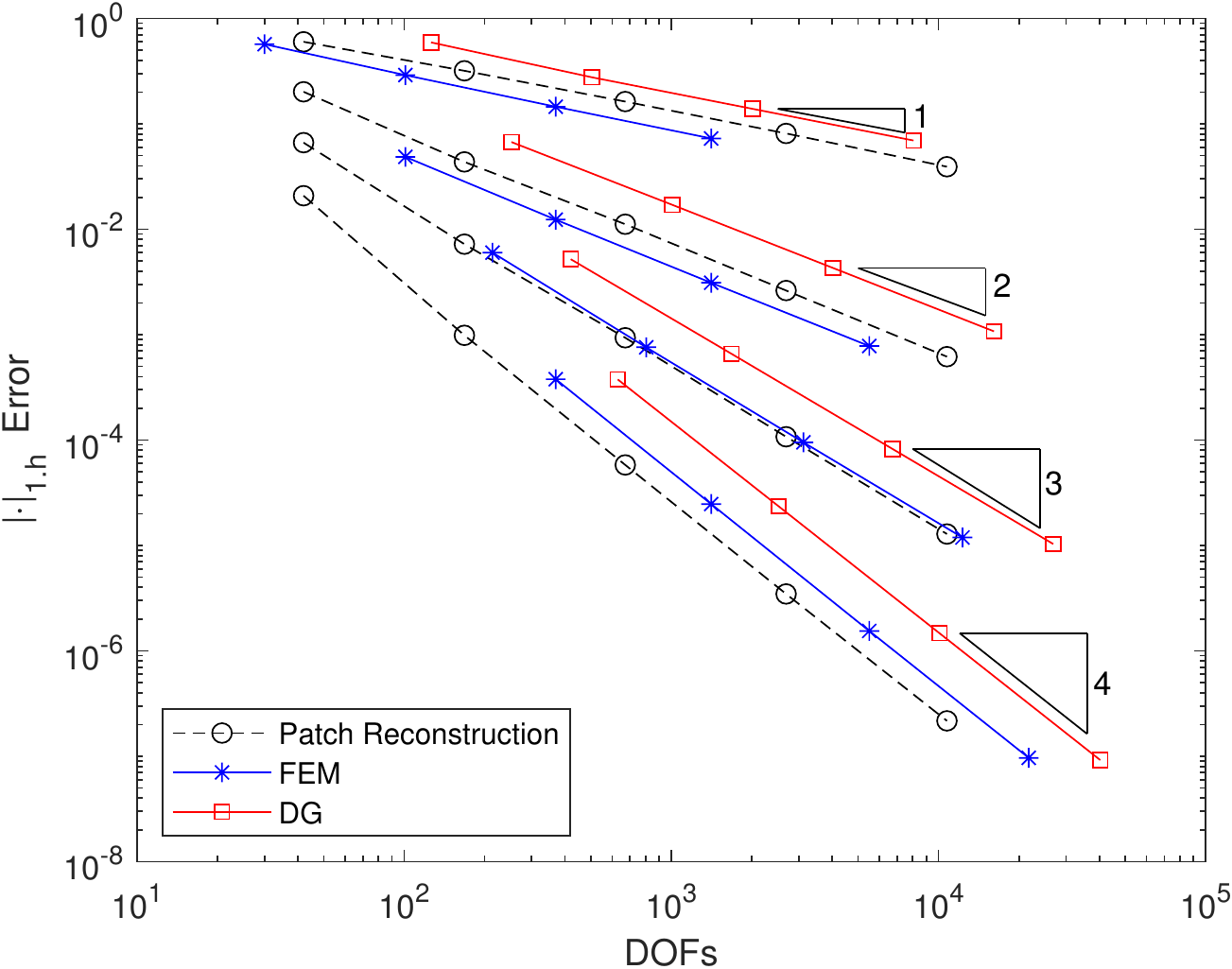}     
\caption{The
convergence rates of the 1st eigenvalue (left) /     eigenfunction
(right) of the second order problem for three     methods on triangle
meshes for Example 1.}     
\label{tri_elliptic_compare}   
\end{center}
\end{figure}

For the biharmonic problem, Figure \ref{tri_biharmonic_compare} shows
the error in terms of number of DOFs of the $C^0$ IPG method, standard
SIPDG method and our method. The approximation order $m$ is taken as
$2$, $3$, and $4$. The convergence rate for eigenvalue is $h^{2(m-1)}$
and the convergence rate is $h^{m-1}$ for eigenfunction which
perfectly agree with the error estimates. The experiments show that
our method performs better than the other methods in all cases. The
advantage of our method in efficiency is more remarkable for higher
order approximation.

\begin{figure}   \begin{center}
\includegraphics[width=0.48\textwidth
]{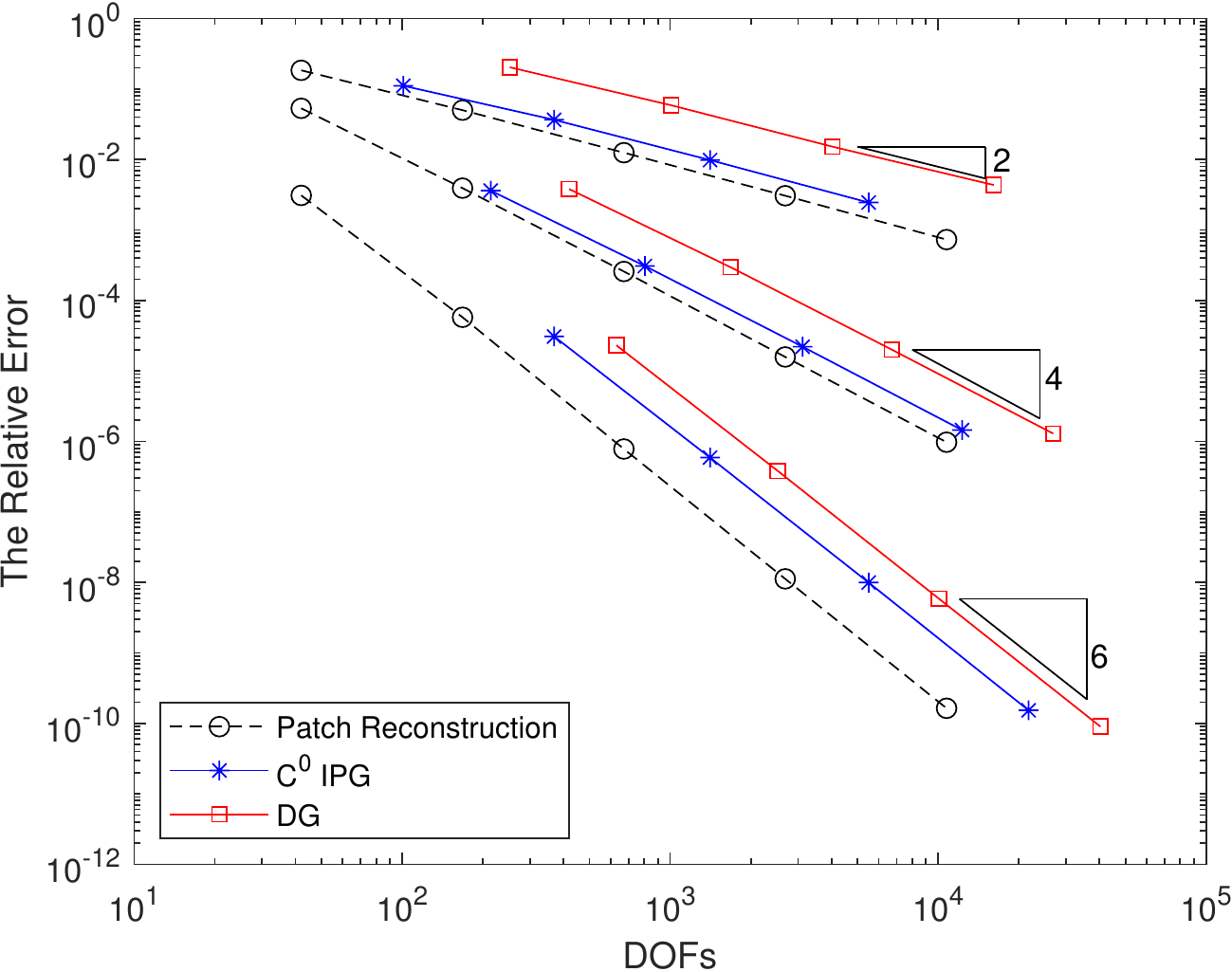}
\includegraphics[width=0.48\textwidth
]{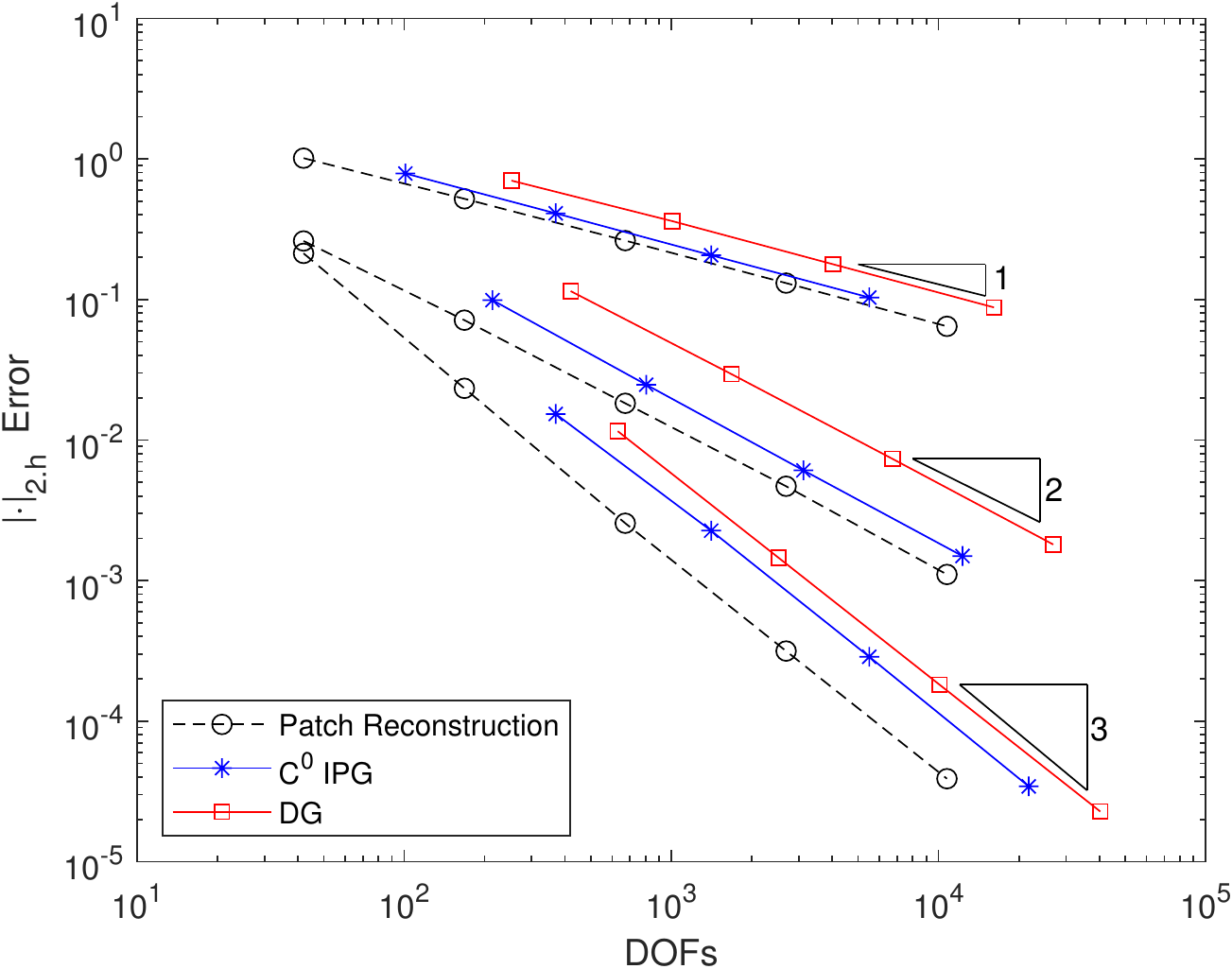}     
\caption{The
convergence rates of the 1st eigenvalue (left) /       eigenfunction
(right) of the biharmonic problem for three       methods on triangle
meshes for Example 1.}     
\label{tri_biharmonic_compare}
\end{center} 
\end{figure}

\subsection{Number of reliable eigenvalues} Zhang studied the number
of reliable eigenvalues of the finite element method in
\cite{zhang2015how}, and the main result he gave is as below:
\begin{theorem}\label{thm:quantity} Suppose that we solve a $2p$-order
elliptic equation on a domain $\Omega \in \mb{R}^D$ by the finite
element method (conforming or non-conforming) of polynomial degree
$m$ under a shape regular and quasi-uniform mesh with mesh-parameter
$h$. Assume that the exact eigenvalue grows as
$\lambda_{j}=O(j^\frac{2p}{D})$ and the relative error can be
estimated by $\frac{\lambda_i^{h}-\lambda_i}{\lambda_i}=
h^{m+1-p}\lambda_{i}^{\frac{m+1}{p-1}}$. Then there are about   \[
j_N=N^\frac{{m+1-p-\alpha/2}}{m+1-p}
m^{-D\frac{m+1-p-\alpha/2}{m+1-p}}   \]   reliable numerical
eigenvalues with the relative error of   $\lambda_{jN}$, converging at
rate $h^{\alpha}$ for $\alpha \in   (0,2(m+1-p)]$. Here $N$ is the
total degrees of freedom. \end{theorem}

Theorem \ref{thm:quantity} implies that the quantity of the reliable
numerical eigenvalues who have the optimal convergence rate
$\alpha=2(m+1-p)$ is $O(1)$, which means only eigenvalues lower in the
spectrum can achieve the optimal convergence rate. Therefore, for the
eigenvalue problem, the number of eigenvalues that have the optimal
convergence rate is very small. We here relax the convergence rate to
linear, saying taking $\alpha=1$, to identify if a numerical
eigenvalue is reliable. For the lowest order approximation of the
eigenvalue problem, linear element for Laplace operator and quadratic
element for biharmonic operator shall be involved. The predicted
number of the reliable numerical eigenvalues from Theorem
\ref{thm:quantity} is $O(N^{1/2})$, which implies that the percentage
of the reliable numerical eigenvalues reduce rapidly as the number of
DOFs of the system increases. For the higher order approximation, the
percentage of the reliable numerical eigenvalues reduces much slower
than the low order approximation.

To identify numerically if an eigenvalue is reliable, we define the
relative error by $\frac{|\lambda - \lambda_h|}{|\lambda|}$, and the
convergence rate by $\log_{2}\left( \frac{|\lambda -
\lambda_{2h}|}{|\lambda - \lambda_h|} \right)$. If the convergence
rate is not less than $1$, the eigenvalue is identified as reliable.
We carry out a series of numerical experiments with various $m$, while
the results are quite robust with almost the same efficiency.

\begin{table} \begin{center} \scalebox{1.0}{
\begin{tabular}{|c|c|c|c|c|} \hline Order & $N$(\#DOF) & 242 & 1,046 &
4,278 \\ \hline $m=1$ & \multirow{4}{*}{Laplace} &8 (3.3\%) & 17
(1.6\%) & 39 (0.9\%) \\ \cline{1-1} \cline{3-5}$ m=2$ & &32 (13.2\%) &
92 (8.8\%) & 270 (6.3\%) \\ \cline{1-1} \cline{3-5} $m=3$ & &38
(15.7\%) & 147 (14.0\%) & 553 (12.9\%) \\ \cline{1-1} \cline{3-5}
$m=4$ & & 96 (39.6\%) & 355 (33.9\%) & 1417 (33.1\%) \\ \hline \hline
$m=2$ & \multirow{3}{*}{Biharmonic} &24(9.9\%) & 53(5.0\%) & 94(2.2\%)
\\ \cline{1-1} \cline{3-5} $m=3$ & &45(18.6\%) & 204(19.5\%) &
691(16.1\%) \\ \cline{1-1} \cline{3-5} $m=4$ & &170 (70.2\%) & 705
(67.3\%) & 2798 (65.4\%) \\ \hline \end{tabular}} \caption{The number
$j_N$ of linear converged eigenvalues.} \label{table_quantity}
\end{center} \end{table}

Again we are limited to study the setup in Example 1 since we need
reference solutions. We calculate $j_{N}$ eigenvalues whose relative
errors are of order $O(h)$. Precisely, we enumerate the number of the
eigenvalues that are at least linearly convergent, with the result
given in Table \ref{table_quantity}. For the Laplace problem, there
are $O(N^{1/2})$ reliable numerical eigenvalues. In this table,  the
percentage decreases rapidly as the computational scale $N$ increases.
The number of the eigenvalues that are at least linearly convergent
increases a lot if the higher order approximation is applied, which is
as implied by Zhang's result that the higher order method could
produce more reliable numerical eigenvalues with the same $N$.
Moreover, for the higher order method, the percentage of the reliable
numerical eigenvalues reduces much slower than the lower order method.

The behavior of the number of reliable eigenvalues is similar for the
biharmonic equation, as shown in Table \ref{table_quantity}. The
numerical results confirm the prediction of Theorem \ref{thm:quantity}
and emphasize that the higher order approximations are more robust and
preferred for the eigenvalue problem.


\section{Conclusion}
We applied the symmetric interior penalty discontinuous Galerkin
method based on a patch reconstructed approximation space for solving
elliptic eigenvalue problem. The proposed method, when
  compared to other existing approximation methods, can be implemented
  in a more flexible way and its approximation properties are easier
  to analyse. Numerical results confirm the optimal convergence
rates and emphasize the great efficiency of our method in number of
DOFs.  The great efficiency and convenient implementation
  is even remarkable in the case of higher order approximation. Since
  high order approximation is preferred for the elliptic eigenvalue
  problems, our method is a quite appropriate method to solve the
  elliptic eigenvalue problems.



\section*{Acknowledgment}
The authors would like to thank the anonymous referees. They have very
constructively helped to improve the original version of this paper.
 
The research is supported by the National Natural Science Foundation
of China (Grant No.  91630310, 11421110001 and 11421101) and Science
Challenge Project, No.TZ2016002.


\end{document}